\documentclass[11pt,reqno]{amsart}

\usepackage{times}
\usepackage{amsmath,amsfonts, amstext,amssymb,amsbsy,amsopn,amsthm}
\usepackage[initials,alphabetic]{amsrefs}
\usepackage{mathrsfs}
\usepackage{bm}
\usepackage{dsfont}
\usepackage{esint}
\usepackage{graphicx}   
\usepackage{hyperref}
\usepackage[all]{xy}
\usepackage{relsize}
\usepackage{mathtools}

\newtheorem{theorem}{Theorem}[section]

\newtheorem{proposition}[theorem]{Proposition}
\newtheorem{lemma}[theorem]{Lemma}

\theoremstyle{definition}
\newtheorem{definition}[theorem]{Definition}
\theoremstyle{remark}
\newtheorem{remark}{Remark}[section]
\theoremstyle{remark}

\theoremstyle{remark}

\newtheorem{claim}{Claim}[section]
\theoremstyle{remark}

\theoremstyle{remark}

\theoremstyle{remark}

\theoremstyle{remark}

\numberwithin{equation}{subsection}

\begin{document}

\title{Quantitative stratification of stationary connections}
\date{\today}
\author{Yu Wang}
\address{Department of Mathematics, Northwestern University, Evanston, Il 60208, USA}
\email{yuwang2018@u.northwestern.edu}

\maketitle
\begin{abstract}
Let $A$ be a connection of a principal bundle $P$ over a Riemannian manifold $M$, such that its curvature $F_A\in L_{\text{loc}}^2(M)$ satisfies the stationarity equation. It is a consequence of the stationarity that $\theta_A(x,r)=e^{cr^2}r^{4-n}\int_{B_r(x)}|F_A|^2$ is monotonically increasing in $r$, for some $c$ depending only on the local geometry of $M$. We are interested in the singular set defined by $S(A)=\{x: \lim_{r\to 0}\theta_A(x,r)\neq 0\}$, and its stratification $S^k(A)=\{x: \text{no tangent measure at $x$ is $(k+1)$-symmetric}\}$. We then introduce and study the quantitative stratification $S^k_{\epsilon}(A)$. Roughly speaking, $S^k_{\epsilon}(A)$ consists of points at which no tangent measure of $A$ is $\epsilon$-close to being $(k+1)$-symmetric. In the main Theorem, we show that $S^k_{\epsilon}$ is $k$-rectifiable and satisfies the Minkowski volume estimate $\text{Vol}(B_r(S^k_{\epsilon})\cap B_1)\le Cr^{n-k}$. Lastly, we apply the main theorems to the stationary Yang-Mills connections to obtain a rectifiability theorem that extends some previously known results in \cite{T00}.
\end{abstract}

\tableofcontents


\section{Introduction}\label{introsec}
Let $P$ be a principal bundle with compact Lie group fiber $G$ over an $n$-dimensional ($n\ge4$) Riemannian manifold $M$. 
Consider a connection $A\in L^2_{\text{loc}}(M,\mathfrak{g}_P)$ on $P$. Define the curvature of $A$ to be $F_A=dA+\frac{1}{2}[A\wedge A]$ in the distribution sense. Further assume that $F_A\in L_{\text{loc}}^2(M,\mathfrak{g}_P\otimes\Lambda^2M)$. The most important hypothesis in this paper will be the stationarity, defined as follows:
\begin{definition}
\label{sta0}
Let $A$ be a connection satisfying the preceding conditions. We say that $A$ is a stationary connection, if
\begin{equation}
\label{sta}
\int_{M} \bigg{(}|F_A|^2\text{div}X-4\sum_{i,j=1}^n \langle F_A(\nabla_{e_i}X,e_j),F_A(e_i,e_j)\rangle \bigg{)} dV_g=0
\end{equation}
for all smooth vector fields $X$ which are compactly supported in $M$. Here $\{e_i\}_i$ denotes an arbitrary fixed orthonormal frame on $M$.
\end{definition}
While all the main results in this paper hold true and are stated for general Riemannian manifolds, since the results themselves are local, the base manifold geometry is unessential and mainly contributes to unnecessary technicalities. In order to better focus on addressing the key issues and carrying out the main ideas in a clean manner, we may assume that $A$ is a stationary connection of a Principal $G$-bundle defined on $B_{16}(p)\subseteq \mathbb{R}^n$ (often equipped with the Euclidean metric) and satisfies $\int_{B_{16}(p)}|F_A|^2 dV_g\le\Lambda$. Under this simplification, let us adopt the following notation:
\begin{equation}
\theta_A(x,r)=r^{4-n}\int_{B_r(x)}|F_A(x)|^2dV_g.
\end{equation}
It follows from \cite{P83} that the stationarity equation \eqref{sta} yields the following monotonicity formula:
\begin{eqnarray}
\begin{split}
\label{mono}
\theta_A(x,\rho)-\theta_A(x,\sigma)&=\rho^{4-n}\int_{B_\rho(p)}|F_A|^2dV_g-\sigma^{4-n}\int_{B_\sigma(p)}|F_A|^2dV_g\\
&=\int_{A_{\sigma,\rho}(0)}4|x-p|^{4-n}|\iota_{\partial r_p}F_A|^2dV_g.
\end{split}
\end{eqnarray}
Here $A_{\sigma,\rho}(p)=B_{\rho}(p)\backslash B_{\sigma}(p)$, and $\iota_v F=\langle v,F\rangle$. More generally, by \cite{T00}, given any function $\phi(\theta)$ on the unit sphere in $\mathbb{R}^n$ and $\psi(x)=\phi(\frac{x-p}{|x-p|})$, we have:
\begin{eqnarray}
\begin{split}
\label{psisph}
&\rho^{4-n}\int_{B_\rho(p)}\psi|F_A|^2dV_g-\sigma^{4-n}           \int_{B_\sigma(p)}\psi|F_A|^2dV_g\\
=&\int_{A_{\sigma,\rho}(p)}4|x-p|^{4-n}\psi|\iota_{\partial r_p}F_A|^2dV_g-\int_{\sigma}^{\rho}4\tau^{3-n}\big{(}\int_{B_{\tau}(p)}|x-p|\langle \iota_{\partial r_p}F_A,\iota_{\nabla\psi}F_A \rangle  dV_g\big{)} d\tau.
\end{split}
\end{eqnarray}

We now define
\begin{eqnarray}
\begin{split}
S(A)=\big{\{}x: \lim_{r\to 0}\theta_A(x,r)\neq 0\big{\}}.
\end{split}
\end{eqnarray}
Fix any $x_*\in S(A)$, and any positive real number sequence $r_i\to 0$, one could consider the rescaled sequence given by $A_i(x)=r_i^{-1}A(r_i(x-x_*)+x_*)$. Due to the monotonicity formula \eqref{mono}, $|F_{A_i}|^2dV$ weak-$*$ subconverges as measures to some Radon measure $\mu$ of $\mathbb{R}^n$, the tangent plane at $x_*$. Here, $\mu$ is called a tangent measure of $A$ at $x_*$. It is also easily seen from the monotonicity formula \eqref{mono} that $r^{4-n}\mu(B_r(0))$ is nonzero, finite, and constant in $r$. For its proof, we refer the readers to Lemma 3.2.1 and Lemma 5.3.1, \cite{T00}. Now we present the following:
\begin{definition}
\label{symdef1}
Let $\mu$ be a Radon measure of $\mathbb{R}^n$ satisfying $r^{4-n}\mu(B_r(0))\equiv C\in [0,\infty)$. For any $0\le k\le n$, $\mu$ is said to be $k$-symmetric if there exists a $k$-dimensional subspace of $\mathbb{R}^n$ denoted by $V_k$, such that $T_v^* \mu=\mu$ for all $v\in V_k$. Here $T_v$ denotes the translation map $\mathbb{R}^n\to \mathbb{R}^n, x\mapsto x+v$, and $T_v^*$ denotes the pull back operator under $T_v$.
\end{definition}
\begin{remark}
By preceding discussions, every tangent measure of $A$ is $0$-symmetric.
\end{remark}
\begin{definition}\label{classsym}
For $k=0,\cdots, n-1$, define $S^k(A)=\{x: \text{no tangent measure at $x$ is $(k+1)$-symmetric}\}.$
\end{definition}
This is the classical stratification of $S(A)$. Using the standard Federer dimension reduction principle it is not hard to show that $\text{dim}(S^k(A))\le k$ (here and throughout the paper, ``dim'' denotes the Hausdorff dimension). However, little was known about the rectifiability of $S^k(A)$. In this paper, one of our main results is the $k$-rectifiability of $S^k(A)$. This requires studying the quantitative stratification $S^k_{\epsilon}(A)$ and $S^k_{\epsilon,r}(A)$. Roughly speaking, $S^k_{\epsilon}(A)$ is the set of points at which no tangent measure is ``$\epsilon$-close to being $(k+1)$-symmetric'', and heuristically $S^k_{\epsilon,r}(A)$ could be thought of as the $r$-tubular neighborhood of $S^k_{\epsilon}(A)$. Our main results also include the Minkowski volume estimates $\mbox{Vol}(B_r(S_{\epsilon}^{k}(A))\cap B_1(p))\le C(n,\Lambda,\epsilon)r^{n-k}, \mbox{Vol}(B_r(S_{\epsilon}^{k,r}(A))\cap B_1(p))\le C(n,\Lambda,\epsilon)r^{n-k}$. To introduce the definitions of quantitative stratification requires the notion of quantitative symmetry.

\subsection{Quantitative symmetry}
Let us begin by defining ``effectively span'':
\begin{definition}[Effectively span]
\label{effspa}
For fixed $\tau>0$, we say that a $k$-plane $V_k$ is $\tau r$-effectively spanned by $k+1$ points $x_0,\cdots,x_k$ with respect to $B_r(x)$, if $x_0,\cdots,x_k\in B_{r/2}(x)$ and $\text{dist}(x_{i+1},x_0+\text{span}\{x_1-x_0,\cdots,x_i-x_0\})\ge \tau r\ \mbox{for}\ i=0,\cdots,k-1$, where $\text{dist}(\cdot,\cdot)$ stands for the Euclidean distance.
\end{definition}
\begin{remark}
From now on, let us fix $\tau\equiv\tau(n)$ for some constant $\tau(n)>0$. The explicit choice of $\tau(n)$ will be specified in Section \ref{cutoff}.
\end{remark}
\begin{remark}
The notion of effectively span also appears in \cite{L99}, \cite{T00}, and \cite{NV17}.
\end{remark}
In addition, we need the following notion of ``almost cone tip'':
\begin{definition}[Almost cone tip]
\label{conetip}
Fix arbitrary $B_{2r}(x_0)\subseteq B_{16}(p)$.\\
(1) $x_0$ is called an $(\epsilon, r)$-cone tip of $A$, if $|\theta_{A}(x_0,2r)-\theta_{A}(x_0,2\epsilon r)|\le\epsilon.$\\
(2) $x_0$ is called a strict $r$-cone tip of $A$, if $\theta_A(x_0,s)$ is constant in $s$, for $0<s\le 2r$.
\end{definition}
\begin{remark}\label{iotaremark}
By \eqref{mono}, $x$ is an $r$-cone tip of $A$ if and only if
\begin{equation}\label{radiota}
\iota_{\partial r_x}F_A(y)=0,\ \text{for all }y\in B_r(x).
\end{equation} 
Moreover, if $A$ is smooth, \eqref{radiota} implies that $A$ is gauge equivalent to a radially invariant connection in $B_r(x)$. 
\end{remark}
Now we can present the definition of quantitative symmetry:
\begin{definition}[Quantitative symmetry]
\label{keydef}
Fix arbitrary $B_{2r}(x)\subseteq B_{16}(p)$.\\
(1) $A$ is said to be $k$-symmetric in $B_r(x)$, if there is a $k$-plane $V_k$ with $V_k\cap B_{r/10}(x)\neq \emptyset$, such that every $y\in V_k\cap B_r(x)$ is a strict $r$-cone tip of $A$.\\
(2) $A$ is said to be $(k,\epsilon)$-symmetric in $B_r(x)$, if there is a $k$-plane $V_k$ with $V_k\cap B_{r/10}(x)\neq \emptyset$, such that $V_k$ is $\tau(n) r$-effectively spanned by $x_0,x_1,\cdots,\ x_k$ with respect to $B_r(x)$, and that $x_i$ is an $(\epsilon, r)$-cone tip of $A$, for all $i=0,\cdots, k$. 
\end{definition}
\begin{remark}
\label{keyremark}
By \eqref{mono}, $A$ is $k$-symmetric in $B_r(x)$ if and only if:
\begin{eqnarray}
\begin{split}
\label{ksymmequiv}
&\text{There exists a $k$-dimensional plane $V_k$ (with $V_k\cap B_{r/10}(x)\neq \emptyset$) spanned by orthonormal vector fields }\\
&\text{$\nu_1,\cdots,\nu_k$ such that $\iota_{\nu_i}F_A(z)\equiv0,\ \iota_{\partial r_{x}}F_A(z)\equiv0,$ for all $z\in B_r(x)$}.
\end{split}
\end{eqnarray}
Indeed, let us choose $x_1\cdots,x_k\in \partial B_{r/2}(x)\cap V_k$ such that $\nu_i=\nu_i=\frac{x_i-x}{|x_i-x|}$ for each $i=1,\cdots,k$. Then the above equivalence follows from the monotonicity formula \eqref{mono} and the following elementary identity:
\begin{equation}
\label{virtue}
d(x,z)\partial r_x-d(x_0,x_i)\nu_i=d(x_i,z)\partial r_{x_i}, \text{for all }z\in B_r(x).
\end{equation}
Moreover, if $A$ is smooth, \eqref{ksymmequiv} implies that $A$ is gauge equivalent to a connection which only depends on $n-k$ variables and is radial invariant in $B_r(x)$. That is to say, up to a smooth gauge transformation, $A$ could be identified with a connection defined on the sphere $S^{n-k-1}$.
\end{remark}
Next, consider a sequence of stationary connections $\{A_i\}_i$ with $\int_{B_{16}(p)}|F_{A_i}|^2dV\le \Lambda$. As Radon measures, $|F_{A_i}|^2dV$ subconverges to some Radon measure $\mu$ in the weak-$*$ sense. Define $\theta_{\mu}(x,r)=r^{4-n}\mu(B_r(x))$. Then by replicating Definition \ref{conetip} and Definition \ref{keydef}, we have:
\begin{definition}[Almost cone tip]
\label{conetipmu}
Let $\mu$ be a Radon measure obtained as in the preceding paragraph. Fix arbitrary $B_{2r}(x_0)\subseteq B_{16}(p)$. \\
(1) $x_0$ is called an $(\epsilon, r)$-cone tip of $\mu$, if $|\theta_{\mu}(x_0,2r)-\theta_{\mu}(x_0,2\epsilon r)|\le\epsilon.$\\
(2) $x_0$ is called a strict $r$-cone tip of $\mu$, if $\theta_\mu(x_0,s)$ is constant in $s$, for $0<s\le 2r$.
\end{definition}
\begin{definition}[Quantitative symmetry]
\label{keydefmu}
Let $\mu$ be a Radon measure in Definition \ref{conetipmu}. Fix arbitrary $B_{2r}(x)\subseteq B_{16}(p)$.\\
(1) $\mu$ is said to be $k$-symmetric in $B_r(x)$, if there is a $k$-plane $V_k$ with $V_k\cap B_{r/10}(x)\neq \emptyset$, such that every $y\in V_k\cap B_r(x)$ is a strict $r$-cone tip of $\mu$.\\
(2) $\mu$ is said to be $(k,\epsilon)$-symmetric in $B_r(x)$, if there is a $k$-plane $V_k$ with $V_k\cap B_{r/10}(x)\neq \emptyset$, such that $V_k$ is $\tau(n) r$-effectively spanned by $x_0,x_1,\cdots,\ x_k$ with respect to $B_r(x)$, and that $x_i$ is an $(\epsilon, r)$-cone tip of $\mu$, for all $i=0,\cdots, k$. 
\end{definition}
For tangent measures of stationary connections, (1) of Definition \ref{keydefmu} is equivalent to Definition \ref{symdef1}. More precisely, we have the following claim:
\begin{claim}\label{introproof}
Let $\mu$ be a tangent measure of a stationary connection $A$ at point $p$. Then $\mu$ is $k$-symmetric in $B_1(0)\subseteq T_pM$ in the sense of Definition \ref{keydefmu}, if and only if $\mu$ is $k$-symmetric in the sense of Definition \ref{symdef1}.
\end{claim}
The proof of Claim \ref{introproof} will be given in Appendix A.

\subsection{Quantitative stratification}
\label{qssec}
Following \cite{NV17}, and using Definition \ref{keydef}, we introduce the following definitions:
\begin{definition}[Quantitative stratification]\label{quanstrat0}\ \\
(1) For each $\epsilon>0$ and $r<1$, we define
$$S_{\epsilon,r}^{k}(A)= \bigg\{y\in B_1(p):\text{$A$ is not }(k+1,\epsilon)\mbox{-symmetric in $B_s(y)$},\ \text{for all } r\le s\le1 \bigg\}.$$
(2) For each $\epsilon>0$, we define
$$S^k_\epsilon(A)=\bigcap_{r>0}S_{\epsilon,r}^{k}(A)\equiv\bigg\{y\in B_1(p): \text{no tangent measure at $y$ is $(k+1,\epsilon)$-symmetric}\bigg\}.$$
(3) For each $k$, we define
$$S^k(A)\eqqcolon\bigcup_{\epsilon>0}S^k_\epsilon(A)=\bigg\{y\in B_1(p): \text{no tangent measure at $y$ is $k+1$-symmetric}\bigg\}.$$
\end{definition}
\begin{remark}
By Claim \ref{introproof}, (3) of Definition \ref{quanstrat0} is equivalent to Definition \ref{classsym}.
\end{remark}
\begin{remark}
Similarly, by using Definition \ref{keydefmu} one could make the definition of $S^k(\mu)$, $S^k_{\epsilon}(\mu)$, and $S^k_{\epsilon,r}(\mu)$ for a Radon measure $\mu$ obtained as in the paragraph preceding Definition \ref{conetipmu}.
\end{remark}
\begin{remark}
The quantitative stratification was introduced and proved extremely useful for the first time in \cite{CN13a}, where the authors obtained the $L^p$ bounds on the Riemann curvature under certain Ricci curvature assumptions, and achieved better regularity in the Einstein case. Later in \cite{CN13b} they extended the idea to the stationary harmonic maps and minimal currents. Since then, the idea has been used in \cite{CHN13}, \cite{CHN15}, \cite{CNV15}, \cite{FMS15}, \cite{BL15} to prove similar results in the areas of mean curvature flow, harmonic map flow, critical sets of elliptic equations, biharmonic maps, etc.
\end{remark}

\subsection{Main results}
\label{mainresults}
Let $B_{16}(p)\subseteq M^n$ where $M^n$ is a Riemannian manifold with metric $g$. Let $K_M$ be the smallest number that the following hold:
$$|\text{sec}_{B_{16}(p)}|\le K_M^2,$$
$$\text{inj}_{B_{16}(p)}\ge K_M^{-1}.$$
Upon rescaling, we assume that $K_M\le 100^{-1}$. Now let us begin by stating our main theorem for the quantitative stratification $S^k_{\epsilon,r}(A)$:
\begin{theorem}
\label{maintheorem}
Let $A$ be a stationary connection satisfying $\int_{B_{16}(p)}|F_A|^2dV_g\le \Lambda$, then for each $k$ and $\epsilon$, there exists $C(n,\Lambda,\epsilon)$ such that for all $r>0$ we have:
\begin{equation}
\label{maininequality}
\mbox{Vol}(B_r(S_{\epsilon,r}^{k}(A))\cap B_1(p))\le C(n,\Lambda,\epsilon)r^{n-k}.
\end{equation}
\end{theorem}
When we study the stratum $S^k_{\epsilon}$, we can refine the above to obtain structure results on the set itself. For the definition of $k$-rectifiability, we refer the readers to \cite{M95}.
\begin{theorem}
\label{maintheorem1}
Let $A$ be a stationary connection satisfying $\int_{B_{16}(p)}|F_A|^2dV_g\le \Lambda$, then for each $k$ and $\epsilon$ there exists $C(n,\Lambda,\epsilon)$ such that for all $r>0$ we have:
\begin{equation}
\label{maininequality1}
\mbox{Vol}(B_r(S_{\epsilon}^{k}(A))\cap B_1(p))\le C(n,\Lambda,\epsilon)r^{n-k}.
\end{equation}
In particular, we have $H^k(S^k_{\epsilon}(A))\le C(n,\Lambda,\epsilon)$. Furthermore, $S^k_{\epsilon}(A)$ is $k$-rectifiable, and for $k$-a.e. $x\in S_{\epsilon}^k(A)$, there exists a unique $k$-plane $V_k\subseteq T_x M$ such that every tangent measure of $A$ at $x$ is $k$-symmetric with respect to $V_k$.
\end{theorem} 
Finally, we close this subsection by stating our main results when it comes to the classical stratification $S^k(A)$. The following theorem follows easily from the previous theorem in view of the formula $S^k(A)=\bigcup_\epsilon S^k_{\epsilon}(A)$.
\begin{theorem}
\label{maintheorem2}
Let $A$ be a stationary connection satisfying $\int_{B_{16}(p)}|F_A|^2dV_g\le \Lambda$, then $S^k(A)$ is $k$-rectifiable for each $k$, and for $k$-a.e. $x\in S^k(A)$, there exists a unique $k$-plane $V_k\subseteq T_x M$ such that every tangent measure of $A$ at $x$ is $k$-symmetric with respect to $V_k$.
\end{theorem}
We end this section by giving some further remarks on a few applications of these results. For convenience denote by $\mathcal{A}(\Lambda)$ the class of connections considered in this paper which satisfy $A\in L^2_{\text{loc}}(B_{16}(p))$, $\int_{B_{16}(p)}|F_A|^2\le \Lambda$, and the stationarity condition \eqref{sta}.

Note that instead of considering the classical singular set of $A$
\begin{equation}
\text{Sing}(A)=\{x: \text{there exists a neighborhood of $x$ in which $A$ is regular}\},
\end{equation}
the entire paper is devoted to studying $S(A)=\{x: \lim_{r\to 0}\theta_A(x,r)\neq 0\}$. While it is unlikely that $\text{Sing}(A)$ agrees with $S(A)$ for a general $A\in \mathcal{A}(\Lambda)$, one can prove $\text{Sing}(A)=S(A)$ by imposing further regularity assumptions to restrict to a subclass of $\mathcal{A}(\Lambda)$. For instance, in \cite{TT04} the authors added the additional assumptions that $A$ being admissible Yang-Mills, and proved the following $\epsilon$-regularity theorem:
\begin{theorem}[\cite{TT04}]
Let $A$ be an admissible stationary Yang-Mills connection with $\int_{B_2(p)}|F_A|^2\le \Lambda$. Then there exists $\epsilon\le\epsilon(n,\Lambda)$, for all $x\in B_1(p),r\le 2$, if $\theta_A(x,r)\le \epsilon$, then $A$ is smooth in $B_{r/2}(x)$.
\end{theorem}
It is an immediate consequence of this theorem that $\text{Sing}(A)=S(A)$. Very recently, the authors of \cite{PR17} proved a similar $\epsilon$-regularity theorem for a different subclass of $\mathcal{A}(\Lambda)$; see Definition 1.7 and Theorem 1.16 of \cite{PR17}. Therefore, by applying Theorem \ref{maintheorem2} to the connections $A$ considered in \cite{TT04} and \cite{PR17}, we immediately obtain the rectifiability results about their classical singular sets $\text{Sing}(A)$.

\subsection{Quantitative stratification of stationary harmonic maps}\label{qsshm}
In the earlier pioneering work \cite{NV17}, the authors studied the quantitative stratifications of stationary harmonic maps. They first defined the quantitative symmetry of maps. Using this they defined the quantitative stratification $S^k_{\epsilon,r}(f)$ of any stationary harmonic map $f$. Their main results include the Minkowski volume estimates $\mbox{Vol}(B_r(S_{\epsilon,r}^{k}(f))\cap B_1(p))\le C(n,\Lambda,\epsilon)r^{n-k}$, $\mbox{Vol}(B_r(S_{\epsilon}^{k}(f))\cap B_1(p))\le C(n,\Lambda,\epsilon)r^{n-k}$ as well as the rectifiability of $S^k_{\epsilon}(f)$ and $S^k(f)$; see Theorems 1.3, 1.4 and 1.5 of \cite{NV17}.

This paper is originally motivated by \cite{NV17}, with the intention of proving results similar to Theorems 1.3, 1.4 and 1.5 of \cite{NV17} in the setting of stationary connections.

Therefore, there are many similar aspects shared by \cite{NV17} and this paper, which include the important role of a monotone quantity, the same type of problems (see Subsection \ref{qsshm}), as well as the main technical tools used in tackling the problems, i.e. the rectifiable-Reifenberg theorem (see Section \ref{recreifen}) and the $L^2$-best approximation theorem (see Section \ref{l2apth}). However, this paper sees new difficulties.

\subsection{New difficulties}\label{newdiff}
In this subsection, we first point out the difficulties in generalizing the quantitative stratification theory from the context of stationary harmonic maps in \cite{NV17} to the current context of stationary connections. Then we explain the key ideas introduced in this work that overcome them and further strengthens \cite{NV17}.

The first real challenge lies in extending the notion of quantitative symmetry to the context of stationary connections, in order to produce a satisfying quantitative stratification theory. Indeed, one could naturally come up with a definition similar to Definition 1.1 of \cite{NV17} in the context of stationary connections, as shown below:
\begin{definition}[An analogue of Definition 1.1 of $\text{\cite{NV17}}$]
\label{fakedef}
$A$ is $(k,\epsilon)$-symmetric in $B_r(x)$ if there exists a gauge transformation $\sigma$ defined on $B_r(x)$ and a $k$-symmetric $L^2$-connection $\tilde{A}$ such that
\begin{equation}
\label{fakedefeq}
r^{2-n}\int_{B_r(x)}|\sigma^* A-\tilde{A}|^2\le\epsilon.
\end{equation}
Here, the connection forms $\sigma^*A$ and $\tilde{A}$ play the roles of the maps $f$ and $\tilde{f}$ in Definition 1.1 of \cite{NV17} respectively. This correspondence could be explained by a well known heuristic that the connection $A$ acts like an antiderivative of the curvature $F_A$.
\end{definition}
However, if we work with Definition \ref{fakedef} we will face extra technical issues such as lacking Uhlenbeck compactness. This will be highly problematic in arguing by contradiction which is common in the paper. For example, in such arguments it will often be the case that we consider a sequence of connections $A_i$ being $(k,\epsilon_i)$-symmetric with $\epsilon_i\to 0$, which also violates the desired conclusion. From $\{A_i\}_i$ we hope to extract a subsequence (via ``compactness'') that nicely converges to some strict $k$-symmetric connection which is supposed to satisfy the desired conclusion. Thus, upon passing to the limit a contradiction occurs. Unfortunately, if we do not impose a strong enough Sobolev control (or other possible regularity assumptions) on $A$, our only assumption $|F_A|\in L^2_{\text{loc}}$ would be too weak to enforce an Uhlenbeck compactness, and the proposed contradiction arguments break down. For detailed discussions on the Uhlenbeck compactness, please see Chapters 6, 9, and 10 of \cite{W04}; see also \cite{U82a} and \cite{U82b}.

In order to avoid the issues stated above, we make Definition \ref{keydef} instead of directly generalizing Definition 1.1 of \cite{NV17}. The benefits from the new definition are significant. Firstly, since the curvature $F_A$ is used in place of the connection form, the notion of $(k,\epsilon)$-symmetry is gauge invariant. As a matter of fact, we are saved using the gauge transformations throughout this paper. Secondly, instead of Uhlenbeck compactness, the weak-$*$ compactness of positive Radon measures with uniformly bounded variations now becomes sufficient. Thirdly, from the quantitative stratification theoretical viewpoint, this definition is adapted to various geometric contexts, such as harmonic maps, mean curvature flows, minimal currents, etc. Indeed, to generalize Definition \ref{keydef} to a different context, one only needs to replace $\theta_A(x,r)$ by the monotone quantity in that context. For future convenience, let us refer to the quantitative symmetry defined in this way as ``$\theta$-type'' (here ``$\theta$'' stands for the monotone quantity). Correspondingly, refer to Definition 1.1 of \cite{NV17}, Definition \ref{fakedefeq}, etc., as ``$L^2$-type''. In addition, denote by $S^k_{\delta,r,\theta}$ (resp. $S^k_{\delta,r,L^2}$) the quantitative stratifications defined by $\theta$-type (resp. $L^2$-type) quantitative symmetry.

Despite the positive aspects listed above, a new trouble is introduced. Roughly speaking, $\theta$-type is a stronger notion than $L^2$-type. While this is only a heuristic for connections (rigidifying it would require more regularity assumptions on $A$), it could be made rigorous in the context of stationary harmonic maps: let $f$ be a stationary harmonic map between Riemannian manifolds $M$ and $N$ with $\int_{M}|\nabla f|^2dV_g\le \Lambda$. For the sake of convenience, we assume that $M=B_{16}(p)$ equipped with Euclidean metric. It is well known that $\theta_f(x,r)=r^{2-n}\int_{B_r(x)}|\nabla f|^2 dV$ satisfies a monotonicity formula (e.g. see \cite{L99}, \cite{NV17}, etc,), which allows us to define the $\theta$-type quantitative symmetry of $f$. The heuristic that $\theta$-type being stronger than $L^2$-type could now be made rigorous in the context of stationary maps as follows:
\begin{claim}
\label{weakstrong}
For all $\epsilon>0$ and $k=0,\cdots,n$, there exists $\delta(n,\epsilon,\Lambda)$ such that if $f$ is $\theta$-type $(k,\delta)$-symmetric in $B_r(x)$, then $f$ is $L^2$-type $(k,\epsilon)$-symmetric in $B_r(x)$. In addition, $S^k_{\epsilon,r,L^2}(f)\subseteq S^k_{\delta,r,\theta}(f)$.
\end{claim}
By using Poincar\'e inequality it is not hard to check the above claim, and we omit the details. However it is worth noting that the two types of definitions are equivalent in the case of minimizing harmonic maps due to its sequential compactness; see \cite{S96} and \cite{SU82}. On the other hand, by following exactly the same lines of proofs in this paper, we can prove:
\begin{theorem}
\label{nvstrong}
Consider a stationary harmonic map $f:B_{16}(p)\subseteq M\to N$ with $\int_{B_{16}(p)}|\nabla f|^2\le \Lambda$, then for all $\epsilon>0$, $k=0,\cdots,n-1$ there exists $C(M,N,\Lambda,\epsilon)$ such that $\text{Vol}(B_r(S^k_{\epsilon,r,\theta}(f)))\le C(M,N,\Lambda,\epsilon)r^{n-k}$. 
\end{theorem}
\begin{remark}\label{extradiff}
From Claim \ref{weakstrong}, we see Theorem \ref{nvstrong} strengthens the conclusion of Theorem 1.3 of \cite{NV17}. Meanwhile, this causes new technical difficulties; more detailed discussions on this will be given in the next subsection as well as Section \ref{l2apth}.  
\end{remark}

\subsection{Outline of the proof of Theorems \ref{maintheorem}, \ref{maintheorem1}, and \ref{maintheorem2}}\label{outlinerl}
Recall our main task is to obtain both the volume estimates of $\text{Vol}(B_r(S^k_{\epsilon,r}))$ and the rectifiability of $S^k_{\epsilon}$. For these purposes we shall apply the rectifiable-Reifenberg theorem, an original and difficult work established in \cite{NV17}, as well as an $L^2$-best approximation theorem, which allows us to apply the rectifiable-Reifenberg theorem. Now let us elaborate them in details. Roughly speaking, the rectifiable-Reifenberg theorem allows one to obtain the $H^k$-measure control together with the $k$-rectifiability of a set $S$, by assuming for $H^k$-a.e. $x\in S$, the scaling invariant $L^2$-distance between $S\cap B_{r}(x)$ and $L^k\cap B_r(x)$ (for some $k$-plane $L^k$) is summable over all dyadic scales $r=2^{-\alpha}$, and the sum itself, as a function on $S$, is small in the $H^k\rvert_S$-integral average sense. More precisely, let us first define $D^k_S(x,r)=\inf_{L^k\subseteq \mathbb{R}^n}r^{-(k+2)}\int_{B_r(x)}d^2(y,L^k)dH_S^k(y)$. The rectifiable-Reifenberg theorem says that if we know
\begin{equation}
\label{allscale0}
\int_{S\cap B_r(x)}\bigg{(}\int_0^rD_{H^k_S}^k(y,s)\frac{ds}{s}\bigg{)}\ dH^k(y)\le \delta(n)^2r^k,
\end{equation}
for some small $\delta(n)$ and all $B_r(x)\subseteq B_2(p)$, then we have $H^k(S\cap B_r(x))\le (1+\epsilon)\omega_kr^k$ for all $x\in S\cap B_1(p)$, and further $S\cap B_1(p)$ is $k$-rectifiable; see Theorem \ref{RR}. The authors of \cite{NV17} also proves a discrete version of rectifiable-Reifenberg which applies to discrete Dirac measures $\mu=\sum_jr_j^k\delta_{x_j}$, which says if one knows
\begin{equation}
\label{allscale00}
\sum_{\alpha\in \mathbb{N}^+,2^{-\alpha}\le 2r}\int_{B_r(x)}D_{\mu}^k(y,2^{-\alpha}) d\mu(y)\le \delta^2r^k.
\end{equation}
then $\sum_j r_j^k\le D(n)$; see Theorem \ref{DRR}. Clearly, in order to achieve the Minkowski volume estimates by using Theorems \ref{RR} and \ref{DRR}, we need to verify that \eqref{allscale0} and \eqref{allscale00} hold for the quantitative stratifications $S^k_{\epsilon}$, $S^k_{\epsilon,r}$ on all balls. In reality the procedure is subtler than what is said here, since we will actually build inductive covers of the quantitative stratifications and apply the discrete rectifiable-Reifenberg theorem to the discrete Dirac measures associated to these covers at each stage in order to keep track of the volume estimates. But for now let us focus on how to control $D_{\mu}^k(y,s)$ in order to apply Theorems \ref{RR}, \ref{DRR}, instead of being concerned about the details of the inductive coverings. In Section \ref{l2apth}, we will prove the so called $L^2$-best approximation theorem (an analogue of Theorem 7.1 of \cite{NV17}) that estimates for us the quantity $D_{\mu}^k(y,s)$, where $\mu$ is a discrete Dirac measure whose support is contained in the target set $S^k_{\epsilon,r}$. Vaguely speaking, the $L^2$-best approximation theorem characterizes how well the support of $\mu$ can be approximated by a $k$-plane in the $L^2$-sense, by using the properties of the connection $A$. More precisely, if $A$ is $(0,\delta)$-symmetric in $B_{8r}(p)$ but not $(k+1,\epsilon)$-symmetric in $B_r(p)$, then the $L^2$-best approximation theorem tells us for all finite measure $\mu$ the following holds:
\begin{eqnarray}
\begin{split}
\label{l2app12}
D_{\mu}^k(p,r)=&r^{-2-k}\inf_{L^k}\int_{B_r(p)}d^2(x,L^k)d\mu(x)\\
\le& C(n,\Lambda,\epsilon)r^{-k}\int_{B_r(p)} |\theta_A(x,8r)-\theta_A(x,\epsilon_1r)|dV_g(z)\\
&+ C(n,\Lambda,\epsilon)r^{2-n-k}\int_{B_r(p)}\int_{B_{\epsilon_1 r}(x)}|\iota_{x-z}F_A|^2(z)dV_g(z)d\mu(x).
\end{split}
\end{eqnarray}
where $\epsilon_1$ could be an arbitrary constant for now, but will be specified later (in Section \ref{cutoff}) such that $\epsilon_1<\epsilon_1(n,\Lambda,\epsilon)$ for some sufficiently small $\epsilon_1(n,\Lambda,\epsilon)$. In view of Definition \ref{conetip}, the $L^2$-best approximation theorem quantitatively generalizes the phenomenon that ``when a cone is not close to being $k+1$ symmetric, then all of its cone tips locate close to some $k$-dimensional plane'' (see Remark \ref{quangene}). Further we point out that as a trade-off to adopting the stronger definition of quantitative symmetry, an extra second term appears on the right hand side of \eqref{l2app12} compared to (7.2) of \cite{NV17}. For the reason why the extra term persists, see Remark \ref{7217}. To illustrate how we obtain the quantitative stratification estimates, let us consider the following set as an easy example
\begin{eqnarray}
\begin{split}
\tilde{S}^k_{\epsilon,r}=\{x\in S^k_{\epsilon,r}\cap B_1(p): \sup_{y\in B_r(x)}\theta_A(y,\eta r)>E-\eta \}.
\end{split}
\end{eqnarray}
where $E=\sup_{y\in B_1(p)}\theta_A(y,1)$. Choose a Vitali cover of $\tilde{S}^k_{\epsilon,r}$ by $\{B_{5r}(x_i)\}_i$ with $\theta_A(x_i,\eta r)>E-\eta$, and then set $\mu=\sum_i r^k\delta_{x_i}$. We will now estimate $\text{Vol}(B_r(\tilde{S}^k_{\epsilon,r}))$ by applying Theorems \ref{l2ap} and \ref{RR} to the measure $\mu$. Before we start let us remark that the estimate of $\text{Vol}(B_r(S^k_{\epsilon,r}))$ follows very similar strategy, though the proof is more complicated involving a partition technique for a cover and an inductive covering construction. For the purposes of our outline, in this subsection we simply carry out the estimate for $\tilde{S}^k_{\epsilon,r}$. For simplicity sake let us set $|\theta_A(x,8r)-\theta_A(x,\epsilon_1r)|=W_r(x)$ (for some small $\epsilon_1$ to be determined later in the proof). We shall prove (inductively on $\alpha$) that $\mu(B_s(x))\le D(n)s^k$ holds on all $B_s(x)$ with $s\le 2^{-\alpha}$. The beginning stage (i.e. $2^{-\alpha_*}=r$) is trivial by the Vitali cover property. Next, assume this holds for all scales $s$ below $2^{-\alpha_0}$ for some $\alpha_0$. We need to show it also holds for $\alpha_0-1$. On the one hand, in view of the discrete rectifiable-Reifenberg, this immediately follows if we could show
\begin{eqnarray}
\begin{split}
\label{drrl2}
&s^{-k}\sum_{2^{-\alpha}\le 2s}\int_{B_s(x)}D_{\mu}^k(y,2^{-\alpha}) d\mu(y)\le \delta(n)^2,\ \text{for all } s\le 2^{-\alpha_0+2}.
\end{split}
\end{eqnarray}
On the other hand, by applying the $L^2$-best approximation theorem with $r=2^{-\alpha}$ for each $2^{-\alpha}\le 2s$, taking sum of \eqref{l2app} over all such $\alpha$s, and then integrating, one has 
\begin{eqnarray}
\begin{split}
\label{drrl3}
&s^{-k}\sum_{2^{-\alpha}\le 2s}\int_{B_s(x)}D_{\mu}^k(y,2^{-\alpha}) d\mu(y)\le \mathcal{W}+\mathcal{E},
\end{split}
\end{eqnarray}
where
\begin{eqnarray}
\begin{split}
\label{drrl4}
&\mathcal{W}=Cs^{-k}\int_{B_s(x)}\sum_{2^{-\alpha}\le s}W_{2^{-\alpha}}(z)d\mu(z),\\ 
&\mathcal{E}=Cs^{-k}\int_{B_s(x)} \sum_{2^{-\alpha}\le s}(2^{-\alpha})^{2-n}\int_{B_{\epsilon_1 2^{-\alpha}}(z)}|\iota_{z-x}F_A(y)|^2 dV_g(y)  d\mu(z).
\end{split}
\end{eqnarray}
In obtaining \eqref{drrl3}, we have used the inductive assumption that $\mu(B_s(x))\le D(n)s^k$ for all $x\in B_1(p)$ and $s\le 2^{-\alpha_0}$, in order to cancel out the $2^{-k\alpha}$ factors on the right hand side of \eqref{l2app12} where we have set $r=2^{-\alpha}$ for each $2^{-\alpha}\le 2s$. Notice that $\mathcal{E}$ arises from the extra second term in \eqref{l2app12}. However, by further exploiting the monotonicity formula \eqref{mono} we show (in \eqref{dyaest} of Section \ref{cutoff}) that the extra second term in \eqref{l2app} is actually summable (over dyadic scales), and further $\mathcal{E}$ could be made less than $\delta(n)^2/2$ by choosing $\epsilon_1\equiv\epsilon_1(n,\Lambda,\epsilon)$ to be small enough from the very beginning of the proof. In estimating $\mathcal{W}$, in the authors of \cite{NV17} made the key observation that the energy drop function $W_{2^{-\alpha}}(x)$ is summable over $\alpha$. Then, by the fact that $\theta_A(x_i,\eta r)>E-\eta$ we can make $\mathcal{W}$ less than $\delta(n)^2/2$ by choosing $\eta<\eta(n,\Lambda,\epsilon)$. We could now conclude \eqref{drrl2} from Theorem \ref{DRR}, and hence finishes the induction. Finally, by taking $\alpha=0$ and using the Vitali cover property, we obtain $\text{Vol}(B_r(\tilde{S}^k_{\epsilon,r}))\le C(n)$ for some $C(n)\gg D(n)$.

To sum up, even though \eqref{l2app12} is weaker than (7.2) of \cite{NV17}, the new error is summable and the sum could be made so small that we could still achieve $r^{-k}\sum_{2^{-\alpha}\le 2r}\int_{B_r(x)}D_{\mu}^k(y,2^{-\alpha}) d\mu(y)\le \delta(n)^2$ in order to apply Theorem \ref{DRR} to obtain the Minkowski volume estimates. This is the most crucial new point of this paper.


\section{Rectifiable-Reifenberg theorem}\label{recreifen}
In this section, we will state the rectifiable-Reifenberg theorems proved in \cite{NV17}, a key tool that allows us to obtain quantitative stratification estimates together with the rectifiability. Let us begin with the following definition:
\begin{definition}
\label{l2flat}
Let $\mu$ be a measure in $B_2$ with $r>0$ and $k\in\mathbb{N}$. We define
$$D^k_{\mu}(x,r)=\inf_{L^k\subseteq \mathbb{R}^n}r^{-(k+2)}\int_{B_r(x)}d^2(y,L^k)d\mu(y)$$
when $\mu(B_r(x))\ge \epsilon_nr^k\equiv(1000n)^{-7n^2}r^k$ and $D^k_{\mu}(x,r)\equiv 0$ otherwise, where the infimum is taken over all $k$-dimensional affine subspaces $L^k\subseteq \mathbb{R}^n$.
\end{definition}
\begin{remark}
In literature $D^k_{\mu}$ is usually referred to as Jone's number $\beta_2$.
\end{remark}

\begin{theorem}[rectifiable-Reifenberg \cite{NV17}]\label{RR}
Let $S\subseteq B_2\subseteq \mathbb{R}^n$ be a $H^k$-measurable set. Set $\mu_S\eqqcolon H^k\rvert_{S}$. If for all $ \epsilon>0$, there exists $\delta(n,\epsilon)$, such that for all $B_r(x)\subseteq B_2$ with $H^k(S\cap B_r(x))\ge \epsilon_nr^k$ the following holds
\begin{equation}
\label{allscale}
\int_{S\cap B_r(x)}\bigg{(}\int_0^rD_{\mu_S}^k(y,s)\frac{ds}{s}\bigg{)}\ dH^k(y)\le \delta^2r^k,
\end{equation}
then the following hold\\
(1) for any $B_r(x)\subseteq B_1$ with $x\in S$, $H^k(S\cap B_r(x))\le (1+\epsilon)\omega_kr^k$,\\
(2) $S\cap B_1$ is $k$-rectifiable.
\end{theorem}
We also present a version of above theorem that is more discrete in nature:
\begin{theorem}[Discrete rectifiable-Reifenberg \cite{NV17}]\label{DRR}
Let $\{B_{r_j}(x_j)\}_{x_j\in S}$ be a collection of disjoint balls and let $\mu=\sum_jr_j^k\delta_{x_j}$ be the associated measure. Then there exists $\delta(n)$ and $D(n)$, such that if for all $B_r(x)\subseteq B_2$ with $\mu(B_r(x))\ge \epsilon_nr^k$ the following holds
\begin{equation}
\label{allscale1}
\sum_{\alpha\in \mathbb{N}^+,2^{-\alpha}\le 2r}\int_{B_r(x)}D_{\mu}^k(y,2^{-\alpha}) d\mu(y)\le \delta^2r^k.
\end{equation}
then we have $\sum_jr_j^k\le D(n).$
\end{theorem}

\section{$L^2$-best approximation theorem}
\label{l2apth}
In this section, we will prove Theorem \ref{l2ap}, the $L^2$-best approximation Theorem, the other key ingredient for proving the main theorems. On the one hand, this is an analogue of Theorem 7.1 of \cite{NV17}, with the same purposes of estimating the quantity $D^k_{\mu}(x,r)$ in order to apply the rectifiable-Reifenberg theorems. As a matter of fact, the structure of this section is similar to that of Section 7 of \cite{NV17}. On the other hand, as we have seen from the outline in Subsection \ref{outlinerl}, subtle differences lie in the estimates. This will be discussed in more details throughout this section. From now on, let connection $A$ be the same as in Theorem \ref{maintheorem}.

\subsection{Energy decay function}
For all $x\in S^k_{\epsilon,r}(A)$ we introduce an energy decay function $W_r(x)$ (similar to that defined in (7.1) of \cite{NV17}) as follows:
\begin{eqnarray}
\begin{split}
\label{w2drop}
W_r(x)\eqqcolon\theta_{A}(x,8r)-\theta_{A}(x,\epsilon_1 r),
\end{split}
\end{eqnarray}
where $\epsilon_1\equiv\epsilon_1(n,\Lambda,\epsilon)$ will be determined in the proof of the main theorem in the next section.

\subsection{$L^2$-best approximation theorem}\label{majordiff}
The main result of this section is the following theorem (c.f. Theorem 7.1 of \cite{NV17}):
\begin{theorem}[$L^2$-best approximation theorem]\label{l2ap}
Let $A$ be a stationary connection with
$$\int_{B_{16r}(p)}|F_A|^2 dV_g\le \Lambda.$$
Then for all $\epsilon$, there exists constants $\delta\equiv\delta(n,\Lambda,\epsilon)$ and $C(n,\Lambda,\epsilon)$, if $K_M\le\delta(n,\Lambda,\epsilon)$, and $A$ is $(0,\delta)$-symmetric in $B_{8r}(p)$ but not $(k+1,\epsilon)$-symmetric in $B_r(p)$, then for all finite measure $\mu$ in $B_r(p)$ the following holds:
\begin{eqnarray}
\begin{split}
\label{l2app}
D_{\mu}^k(p,r)&=r^{-2-k}\inf_{L^k}\int_{B_r(p)}d^2(x,L^k)d\mu(x)\\
&\le C(n,\Lambda,\epsilon)r^{-k}\int_{B_r(p)}\bigg{(}W_r(x)+ r^{2-n}\int_{B_{\epsilon_1 r}(x)}|\iota_{x-z}F_A|^2(z)dV_g(z)\bigg{)}d\mu(x).
\end{split}
\end{eqnarray}
where the infimum is taken over all $k$-dimensional affine planes $L^k\subseteq \mathbb{R}^n$. 
\end{theorem}
\begin{remark}\label{quangene}
When $A$ is a strict cone, the $L^2$-distance on the left hand side \eqref{l2app} could be improved as follows:
\begin{claim}
If $A$ is $0$-symmetric in $B_8(p)$ but not $(k+1,\epsilon)$-symmetric in $B_1(p)$, then there exists a $k$-plane $V$ and $\epsilon_1<\epsilon_1(n,\Lambda,\epsilon)$, such that for all $\lambda>0$ we have
\begin{eqnarray}
\begin{split}
\{x\in B_1(p): d(x,V)^2>\epsilon_1^{-1}\lambda \}\subseteq \{x\in B_1(p): W_{1}(x)>\lambda \}.
\end{split}
\end{eqnarray}
\end{claim}
The proof uses the monotonicity formula \eqref{mono} together with very elementary observations from Euclidean geometry, and we omit the details.
\end{remark}

\subsection{Energy lower bound}
\begin{lemma}\label{l2appp}
Let all conditions be the same as in Theorem \ref{l2ap}; then for all $\epsilon$, there exists $\delta\equiv\delta(n,\Lambda,\epsilon)$ such that if $K_M\le\delta(n,\Lambda,\epsilon)$, and moreover $A$ is $(0,\delta)$-symmetric in $B_{8r}(p)$ but not $(k+1,\epsilon)$-symmetric in $B_r(p)$, then for any $k+1$ dimensional subspace $V$ spanned by orthonormal basis $\nu_1,\cdots,\nu_{k+1}$ the following holds:
\begin{equation}\label{directsmall0}
\sum_{i=1}^{k+1}\int_{B_{4r}(p)}|\iota_{\nu_i}F_A|^2(x)dV_g(x)\ge r^{n-4}\delta.
\end{equation}
\end{lemma}
\begin{remark}\label{7217}
This lemma is a counterpart of Lemma 7.2 of \cite{NV17}. Nevertheless, in place of an estimate similar to (7.3) of \cite{NV17} which may look like
\begin{equation}\label{directsmall00}
\sum_{i=1}^{k+1}\int_{A_{3r,4r}(p)}|\iota_{\nu_i}F_A|^2(x)dV_g(x)\ge r^{n-4}\delta^{\prime},
\end{equation}
we achieve a weaker estimate \eqref{directsmall0}, and this weaker estimate is exactly what leads to the extra second term on the right hand side of \eqref{l2app}. We now elucidate why this difference occurs. Firstly, let us point out that the contradiction argument used in the proof of Lemma 7.2 of \cite{NV17} does not carry over if the quantitative symmetry is $\theta$-type instead. This is because while $L^2$-type quantitative symmetry is preserved under the $L^2\cap H^1_{\text{weak}}$ convergence, $\theta$-type is not. Next, it is not hard to check that strengthening \eqref{directsmall0} to \eqref{directsmall00} is equivalent to proving the following proposition (for the readers' convenience we state it for both stationary connections and stationary maps):
\begin{proposition}\label{fake00}
Let $A$ (resp. $f$) be a stationary connection (resp. map) with $\int_{B_{16}(p)}|F_A|^2<\Lambda$ (resp. $\int_{B_{16}(p)}|\nabla f|^2<\Lambda$). Then for all $\delta^{\prime}$, there exists $\delta$, such that if $A$ (resp. $f$) is $\theta$-type $(0,\delta)$-symmetric in $B_{8r}(p)$ and 
$$\sum_{i=1}^{k+1}\int_{A_{3r,4r}(p)}|\iota_{\nu_i}F_A|^2(x)dV_g(x)\le r^{n-4}\delta,\ (\text{resp. } \sum_{i=1}^{k+1}\int_{A_{3r,4r}(p)}|\nabla_{\nu_i}f|^2(x)dV_g(x)\le r^{n-4}\delta,)$$ 
then 
$$\sum_{i=1}^{k+1}\int_{B_{4r}(p)}|\iota_{\nu_i}F_A|^2(x)dV_g(x)\le r^{n-4}\delta^{\prime},\ (\text{resp. } \sum_{i=1}^{k+1}\int_{B_{4r}(p)}|\nabla_{\nu_i}f|^2(x)dV_g(x)\le r^{n-4}\delta.)$$
\end{proposition}
While we can prove this proposition under the extra assumption of $A$ being a smooth Yang-Mills connection (resp. $f$ being a smooth harmonic map), it is not known whether it holds for the general stationary connections (resp. stationary maps) possibly having singularities; indeed, enough regularity (e.g. at most codim-$5$ (resp. codim-$3$) singularities) appears necessary in the proof. In our current situation in which $|F_A|\in L^2$ (resp. $f\in W^{1,2}$) is the only regularity we have, \eqref{directsmall0} is by far the best estimate we are able to obtain.
\end{remark}

\begin{proof}[Proof of Lemma \ref{l2appp}]
Throughout the proof, for convenience we will fix $r=1$ and $p=0^n\in \mathbb{R}^n$. Let $\delta\equiv\delta(n,\epsilon,\Lambda)$ be a small number and will be specified later. To conclude the lemma, we shall prove the contraposition of Lemma \ref{l2appp}. Namely, our goal is to show that if $A$ is $(0,\delta)$-symmetric in $B_8(0)$, and for some $k+1$ subspace through $p$ spanned by an orthonormal basis given by $\{\partial z_i\}_{i=1}^{k+1}$, such that the following holds:
\begin{equation}\label{directsmall}
\sum_{i=1}^{k+1}\int_{B_4(0)}|\iota_{\partial_{z_i}}F_A|^2(x)dV_g(x)< \delta.
\end{equation}
Then $A$ is $(k+1,\epsilon)$-symmetric in $B_1(0)$. This proves such choice of $\delta$ suffices to conclude lemma \ref{l2appp}. Moreover, by the condition that $A$ is $(0,\delta)$-symmetric in $B_8(0)$, for convenience let us assume w.l.o.g. that
\begin{eqnarray}
\begin{split}
\label{1eqngele}
8^{4-n}\int_{B_8(0)} |F_A|^2dV_g-(8\delta)^{4-n}\int_{B_{8\delta}(0)} |F_A|^2 dV_g\le \delta.
\end{split}
\end{eqnarray}

Let $\xi(x)$ be a smooth function that equals to $1$ on $B_1(0)$ and equals to $0$ outside $B_{1+a}(0)$, where $a\equiv a(n,\Lambda,\epsilon)$ will be determined later. By choosing a nicely behaving $\xi $ we could assume that $|\nabla\xi|\le 2 a^{-1}$. Fix any $i_0\in\{1,\cdots,k+1\}$. Let us denote by $e_{i_0}$ the point on the ${i_0}^{th}$ axis with $|e_{i_0}|=1$. For all $s\le2$ and $\tau\in[0,1]$ we set
$$X_{\tau,s}(x)\eqqcolon \xi(\frac{x+\tau e_{i_0}}{s})\partial z_{i_0}.$$
Obviously for the prescribed choices of $s$ and $\tau$, the vector field $X_{\tau,s}(x)$ is smooth and compactly supported in $B_{16}(0)$. Let us inserting $X_{\tau,s}$ into the stationarity equation \eqref{sta} to obtain
\begin{eqnarray}
\begin{split}
\label{stat}
\int_{B_{16}(0)} \bigg{(}|F_A|^2divX_{\tau,s}-4\sum_{l,j=1}^n \langle F_A(\nabla_{\partial z_{l}}X_{\tau,s},\partial z_{j}),F_A(\partial z_{l},\partial z_{j})\rangle \bigg{)} dV_g=0.
\end{split}
\end{eqnarray}
where $\partial z_{1},\cdots,\partial z_{n}$ forms an orthonormal basis on $\mathbb{R}^n$. Next, we compute
\begin{eqnarray}
\begin{split}
\label{divx}
\text{div}X_{\tau,s}&=\frac{\partial }{\partial z_{i_0}}[\xi(\frac{x+\tau e_{i_0}}{s})]\\
&=\frac{d}{d\tau}[\xi(\frac{x+\tau e_{i_0}}{s})],
\end{split}
\end{eqnarray}
\begin{eqnarray}
\begin{split}
\label{covx}
\nabla_{\partial z_{l}}X_{\tau,s}&=\nabla_{\partial z_{l}}[\xi(\frac{x+\tau e_{i_0}}{s})\partial z_{i_0}]\\
&=s^{-1}\frac{\partial\xi}{\partial z_l}(\frac{x+\tau e_{i_0}}{s})\partial z_{i_0}.
\end{split}
\end{eqnarray}
Insert \eqref{divx} and \eqref{covx} into \eqref{stat}, we obtain
\begin{eqnarray}
\begin{split}
\label{stati}
&\frac{d}{d\tau}\bigg{(}\int_{B_{(1+a)s}(-\tau e_{i_0})} |F_A|^2\xi(\frac{x+\tau e_{i_0}}{s}) dV_g\bigg{)}\\
=&4\sum_{l,j=1}^n\int_{B_{(1+a)s}(-\tau e_{i_0})} s^{-1}\frac{\partial\xi}{\partial z_l}\langle F_A(\partial z_{i_0},\partial z_{j}),F_A(\partial z_{l},\partial z_{j})\rangle  dV_g;
\end{split}
\end{eqnarray}
applying Cauchy-Schwarz inequality to the right hand side of \eqref{stati}, we have the estimate
\begin{eqnarray}
\begin{split}
\label{statbd}
&\bigg{|}\frac{d}{d\tau}\bigg{(}\int_{B_{(1+a)s}(-\tau e_{i_0})} |F_A|^2\xi(\frac{x+\tau e_{i_0}}{s}) dV_g\bigg{)}\bigg{|}\\
\le& C(sa)^{-1}\bigg{(}\int_{B_{(1+a)s}(-\tau e_{i_0})} |\iota_{\partial z_{i_0}}F_A|^2  dV_g\bigg{)}^{1/2}\bigg{(}\int_{B_{(1+a)s}(-\tau e_{i_0})} |F_A|^2  dV_g\bigg{)}^{1/2}\\
\le& C(sa)^{-1}\sqrt{\Lambda}(\int_{B_4(0)} |\iota_{\partial z_{i_0}}F_A|^2  dV_g)^{1/2}\\
\le& C(sa)^{-1}\sqrt{\Lambda\delta}.
\end{split}
\end{eqnarray}
In \eqref{statbd}, let us take $s$ to be $2$ and $\frac{\epsilon}{1+a}$ respectively and then integrate from $0$ to $\frac{1}{2}$ with respect to $d\tau$. Using the fundamental theorem of calculus, we have
\begin{eqnarray}
\begin{split}
\label{edropat2}
\bigg{|}2^{4-n}\int_{B_{2(1+a)}(-\frac{1}{2} e_{i_0})} |F_A|^2\xi(\frac{x+\frac{1}{2} e_{i_0}}{2}) dV_g -2^{4-n}\int_{B_{2(1+a)}(0)} |F_A|^2\xi(\frac{x}{2}) dV_g   \bigg{|}\le C(n,\Lambda)a^{-1}\sqrt{\delta},
\end{split}
\end{eqnarray}
\begin{eqnarray}
\begin{split}
\label{edropatep}
\bigg{|}\epsilon^{4-n}\int_{B_{\epsilon}(-\frac{1}{2} e_{i_0})} |F_A|^2\xi(\frac{1+a}{\epsilon}(x+\frac{1}{2} e_{i_0})) dV_g -\epsilon^{4-n}\int_{B_{\epsilon}(0)} |F_A|^2\xi(\frac{1+a}{\epsilon}x) dV_g   \bigg{|}\le C(n,\Lambda,\epsilon)a^{-1}\sqrt{\delta}.
\end{split}
\end{eqnarray}
On the other hand, by choosing $\delta\le\frac{\epsilon}{100}$ and using the fact that $\xi(x)$ equals to $1$ on $B_1(0)$ and equals to $0$ outside $B_{1+a}(0)$, we have
\begin{eqnarray}
\begin{split}
\label{4eqngele}
(2(1+a))^{4-n}\int_{B_{2(1+a)}(0)} |F_A|^2\xi(\frac{x}{2}) dV_g &\le 8^{4-n}\int_{B_8(0)} |F_A|^2dV_g,\\
(\frac{\epsilon}{1+a})^{4-n}\int_{B_{\epsilon}(0)} |F_A|^2\xi(\frac{1+a}{\epsilon}x) dV_g   & \ge (8\delta)^{4-n}\int_{B_{8\delta}(0)} |F_A|^2 dV_g,\\
2^{4-n}\int_{B_{2(1+a)}(-\frac{1}{2} e_{i_0})} |F_A|^2\xi(\frac{x+\frac{1}{2} e_{i_0}}{2}) dV_g &\ge 2^{4-n}\int_{B_2(-\frac{1}{2} e_{i_0})} |F_A|^2dV_g,\\
\epsilon^{4-n}\int_{B_{\epsilon}(-\frac{1}{2} e_{i_0})} |F_A|^2\xi(\frac{1+a}{\epsilon}(x+\frac{1}{2} e_{i_0})) dV_g &\le \epsilon^{4-n}\int_{B_{\epsilon}(-\frac{1}{2} e_{i_0})} |F_A|^2 dV_g.
\end{split}
\end{eqnarray}
Using the third and fourth inequality in \eqref{4eqngele}, we obtain
\begin{eqnarray}
\begin{split}
\label{passs}
&2^{4-n}\int_{B_2(-\frac{1}{2} e_{i_0})} |F_A|^2dV_g-\epsilon^{4-n}\int_{B_{\epsilon}(-\frac{1}{2} e_{i_0})} |F_A|^2 dV_g\\
\le &2^{4-n}\int_{B_{2(1+a)}(-\frac{1}{2} e_{i_0})} |F_A|^2\xi(\frac{x+\frac{1}{2} e_{i_0}}{2}) dV_g -\epsilon^{4-n}\int_{B_{\epsilon}(-\frac{1}{2} e_{i_0})} |F_A|^2\xi(\frac{1+a}{\epsilon}(x+\frac{1}{2} e_{i_0})) dV_g \\
\le &\bigg{|}2^{4-n}\int_{B_{2(1+a)}(-\frac{1}{2} e_{i_0})} |F_A|^2\xi(\frac{x+\frac{1}{2} e_{i_0}}{2}) dV_g -2^{4-n}\int_{B_{2(1+a)}(0)} |F_A|^2\xi(\frac{x}{2}) dV_g   \bigg{|}\\
+&\bigg{|}\epsilon^{4-n}\int_{B_{\epsilon}(-\frac{1}{2} e_{i_0})} |F_A|^2\xi(\frac{1+a}{\epsilon}(x+\frac{1}{2} e_{i_0})) dV_g -\epsilon^{4-n}\int_{B_{\epsilon}(0)} |F_A|^2\xi(\frac{1+a}{\epsilon}x) dV_g   \bigg{|}\\
+&\bigg{|}2^{4-n}\int_{B_{2(1+a)}(0)} |F_A|^2\xi(\frac{x}{2}) dV_g-\epsilon^{4-n}\int_{B_{\epsilon}(0)} |F_A|^2\xi(\frac{1+a}{\epsilon}x) dV_g  \bigg{|} \\
=& I+II+III.
\end{split}
\end{eqnarray}
Using \eqref{edropat2} and \eqref{edropatep}, we have
\begin{eqnarray}
\begin{split}
\label{1and2}
I+II\le C^{\prime}(n,\epsilon,\Lambda)a^{-1}\sqrt{\delta}.
\end{split}
\end{eqnarray}
We then use triangle inequality to estimate $III$ trivially as follows:
\begin{eqnarray}
\begin{split}
III&\le \bigg{|}2^{4-n}\int_{B_{2(1+a)}(0)} |F_A|^2\xi(\frac{x}{2}) dV_g-(2(1+a))^{4-n}\int_{B_{2(1+a)}(0)} |F_A|^2\xi(\frac{x}{2}) dV_g \bigg{|} \\
&+ \bigg{|}(2(1+a))^{4-n}\int_{B_{2(1+a)}(0)} |F_A|^2\xi(\frac{x}{2}) dV_g-(\frac{\epsilon}{1+a})^{4-n}\int_{B_{\epsilon}(0)} |F_A|^2\xi(\frac{1+a}{\epsilon}x) dV_g   \bigg{|}\\
&+ \bigg{|}(\frac{\epsilon}{1+a})^{4-n}\int_{B_{\epsilon}(0)} |F_A|^2\xi(\frac{1+a}{\epsilon}x) dV_g -\epsilon^{4-n}\int_{B_{\epsilon}(0)} |F_A|^2\xi(\frac{1+a}{\epsilon}x) dV_g  \bigg{|}\\
&=IV+V+VI.
\end{split}
\end{eqnarray}
Clearly, 
\begin{eqnarray}
\begin{split}
\label{4and6}
IV+VI\le C(n)\Lambda\bigg{(}|(1+a)^{n-4}-1|+|(1+a)^{4-n}-1|\bigg{)}.
\end{split}
\end{eqnarray}
By \eqref{1eqngele} we obtain:
\begin{eqnarray}
\begin{split}
\label{term5}
V\le  8^{4-n}\int_{B_8(0)} |F_A|^2dV_g-(8\delta)^{4-n}\int_{B_{8\delta}(0)} |F_A|^2 dV_g\le \delta.
\end{split}
\end{eqnarray}
Firstly, by \eqref{4and6} we could choose $a(n,\Lambda,\epsilon)$ to be so small that $IV+VI\le \epsilon/3$, and then by \eqref{1and2} and \eqref{term5} let us choose $\delta(n,\Lambda,\epsilon,a)$ to be so small that $I+II+V\le \epsilon/3$; therefore the left hand side of \eqref{passs} less than $\epsilon$. Now that $i_0$ is arbitrarily chosen from $1,\cdots,k+1$, we have actually proved that $0,-\frac{1}{2}e_1,\cdots,-\frac{1}{2}e_{k+1}$ are all $(\epsilon,1)$-cone tips of $A$; further, they $\frac{1}{2}$-effectively (and hence $\tau(n)$-effectively) span a $k+1$-plane with respect to $B_1(0)$. By definition $B_1(0)$ is $(k+1,\epsilon)$-symmetric. This completes the proof of lemma \ref{l2appp}.
\end{proof}

\subsection{Proof of Theorem \ref{l2ap}}

\begin{proof}[Proof of Theorem \ref{l2ap}]
The proof basically follows Subsection 7.2 of \cite{NV17}. For convenience we assume w.l.o.g. that $r=1$. In addition, up to normalization we assume that $\mu$ is a probability measure on $B_1(p)$. Consider the linear transformation on $\mathbb{R}^n$ defined as follows: 
$$T:\mathbb{R}^n\longrightarrow\mathbb{R}^n,$$
$$v\mapsto\int_{B(p,1)}\langle x,v\rangle xd\mu(x).$$
Clearly, this is a symmetric linear transformation and hence is diagonalizable. Let its (ordered) eigenvalues be $\lambda_1\ge\lambda_2\ge\cdots\ge\lambda_n$, and the corresponding eigenvectors be $\nu_1,\nu_2,\cdots,\nu_n$.
Using Lagrange multiplier method one could show
\begin{eqnarray}
\begin{split}
\label{dist}
\inf_{L^l\subseteq \mathbb{R}^n}\int_{B(p,1)}d^2(x,L^l)d\mu(x)=\lambda_{l+1}+\cdots+\lambda_n\le (n-l)\lambda_{l+1},
\end{split}
\end{eqnarray}
for any $l=1,2,\cdots,n-1$; here the infimum is taken over all $l$-dimensional affine subspaces. For the proof of \eqref{dist}, see Lemma 7.5 of \cite{NV17}. Therefore, to obtain the inequality \eqref{l2app} it suffices to estimate $\lambda_{k+1}$. For convenience we assume from now on that $p=0$ and the center of $\mu$-mass is 0, i.e. $\int_{B_1(0)}xd\mu(x)=0$. Note for all $j=1,\cdots,n$ we have
$$\lambda_j\nu_j=\int_{B_1(0)}\langle x,\nu_j\rangle xd\mu(x).$$
Hence, for all $z\in B_4(0)$ one has
$$\lambda_j\nu_j=\int_{B_1(0)}\langle x,\nu_j\rangle (x-z)d\mu(x).$$
Now we apply the linear transformation $\nu\mapsto \iota_{\nu}F_A(z)$ to both sides of above identity to obtain:
\begin{eqnarray}
\begin{split}
\label{dist1}
\lambda_j\iota_{\nu_j}F_A(z)=\int_{B_1(0)}\langle x,\nu_j\rangle \iota_{x-z}F_A(z)d\mu(x).
\end{split}
\end{eqnarray}
Applying Cauchy-Schwarz inequality to the right hand side of \eqref{dist1} and then taking absolute values on both sides, we obtain
\begin{eqnarray}
\begin{split}
\label{dist2}
|\lambda_j\iota_{\nu_j}F_A(z)|\le \big{(}\int_{B_1(0)}\langle x,\nu_j\rangle^2 d\mu(x)\big{)}^{1/2}\big{(} |\iota_{x-z}F_A(z)|^2d\mu(x)\big{)}^{1/2}.
\end{split}
\end{eqnarray}
Notice
\begin{eqnarray}
\begin{split}
\label{dist3}
\int_{B_1(0)}\langle x,\nu_j\rangle^2 d\mu(x)=\lambda_j.
\end{split}
\end{eqnarray}
Upon inserting \eqref{dist3} into \eqref{dist2} we have
$$\lambda_j^2|\iota_{\nu_j}F_A(z)|^2\le\lambda_j\int_{B_1(0)} |\iota_{x-z}F_A(z)|^2d\mu(x).$$
Now let us integrate the above identity with respect to the volume form $dV_g$ on $B_4(0)$ to obtain
$$\lambda_j\int_{B_4(0)}|\iota_{\nu_j}F_A|^2(z)dV_g(z)\le \int_{B_1(0)}\int_{B_4(0)}|\iota_{x-z}F_A(z)|^2dV_g(z)d\mu(x).$$
Note in the equation above, instead of integrating on the annulus $A_{3,4}(0)$ as in its counterpart equation (7.25) of \cite{NV17}, we integrate on the entire ball $B_4(p)$ in order to apply Lemma \ref{l2appp} later. Now for each fixed $x\in B_1(0)$, we compute
\begin{eqnarray}
\begin{split}
\label{ann}
&\int_{B_4(0)}|\iota_{x-z}F_A(z)|^2dV_g(z)\\
\le&\int_{B_8(x)}|\iota_{x-z}F_A(z)|^2dV_g(z)\\
\le& \int_{A_{\epsilon_1,8}(x)}|\iota_{\partial_{r_x}}F_A(z)|^2|x-z|^{4-n}|x-z|^{n-2}dV_g(z)+ \int_{B_{\epsilon_1}(x)}|\iota_{x-z}F_A(z)|^2dV_g(z)  \\
=&C(n)\bigg{(}W_1(x)+\int_{B_{\epsilon_1}(x)}|\iota_{x-z}F_A(z)|^2dV_g(z)\bigg{)}.
\end{split}
\end{eqnarray}
Here we used \eqref{mono} in both the penultimate inequality and the last identity. By \eqref{ann} we have
\begin{eqnarray}
\begin{split}
\label{inverse}
\lambda_{k+1}\sum_{i=1}^{k+1}\int_{B_4(0)}|\iota_{\nu_j}F_A(z)|^2dV_g(z)&\le \ \sum_{i=1}^{k+1}\lambda_j\int_{B_4(0)}|\iota_{\nu_j}F_A(z)|^2dV_g(z)\\
&\le C^{\prime\prime}(n)\int_{B_1(0)}\bigg{(}W_1(x)+\int_{B(x,\epsilon_1)}|\iota_{x-z}F_A(z)|^2dV_g(z)\bigg{)}d\mu(x).
\end{split}
\end{eqnarray}
Now let us choose $\delta(n,\Lambda,\epsilon)$ in Theorem \ref{l2ap} as small as $\delta(n,\Lambda,\epsilon)$ in Lemma \ref{l2appp}. Since $A$ is $(0,\delta)$-symmetric in $B_8(p)$, but not $(k+1,\epsilon)$-symmetric in $B_1(p)$, we could then apply Lemma \ref{l2appp} to conclude
\begin{eqnarray}
\begin{split}
\label{desired}
\sum_{i=1}^{k+1}\int_{B_4(0)}|\iota_{\nu_j}F_A(z)|^2dV_g(z)\ge \delta.
\end{split}
\end{eqnarray}
Inserting \eqref{desired} into \eqref{inverse} and then combining it with \eqref{dist}, we obtain
\begin{eqnarray}
\begin{split}
&\inf_{L^k\subseteq \mathbb{R}^n}\int_{B_1(0)}d^2(x,L^k)d\mu(x)\\
\le& c(n)\lambda_{k+1}\\
\le&  \frac{c(n)C^{\prime\prime}(n)}{\delta(n,\epsilon,\Lambda)}\int_{B_1(0)}\bigg{(}W_1(x)+\int_{B(x,\epsilon_1)}|\iota_{x-z}F_A(z)|^2dV_g(z)\bigg{)}d\mu(x)\\
\eqqcolon& C(n,\Lambda,\epsilon)\int_{B_1(0)}\bigg{(}W_1(x)+\int_{B(x,\epsilon_1)}|\iota_{x-z}F_A(z)|^2dV_g(z)\bigg{)}d\mu(x)
\end{split}
\end{eqnarray}
Thus we complete the proof of Theorem \ref{l2ap}.
\end{proof}

\section{The inductive covering lemma}\label{cutoff}
In this section, we will prove a key covering lemma, which will later (in Section \ref{pomt}) be iteratively applied in order to complete the proof of the main theorems. Briefly speaking, the lemma allows us to construct inductive coverings of $S^k_{\epsilon,r}$ while keeping effective track of the content estimates through all stages, and finally arrive at a cover of $B_r(S^k_{\epsilon,r})$ by balls with size $r$ whose content estimate will then yield the desired Minkowski volume estimate \eqref{maininequality}. What lies in the heart of proving the lemma is to obtain the effective content estimate of the covers that we build, via applying the $L^2$-best approximation theorem and then the rectifiable-Reifenberg theorems to the measures associated to these covers. Let us begin by stating the Lemma (c.f. Lemma 8.1 of \cite{NV17}):
\begin{lemma}
\label{keylemma}
Let $A$ be stationary connection satisfying the same conditions as Theorem \ref{maintheorem}. Let $E=\sup_{x\in B_1(p)\cap S^k_{\epsilon,r}(A)}\theta_A(x,1)$. Then for all $\eta \le \eta(n,\Lambda,\epsilon)$, there exists a covering $S^k_{\epsilon,r}(A)\cap B_1(p)\subseteq U= U_r\cup U_+$ such that:\\
(1) $U_+=\bigcup_i B_{r_i}(x_i)$ with $r_i> r$ and $\sum_i r_i^k\le C(n,\Lambda, \epsilon)$;\\
(2) $\sup_{y\in B_{r_i}(x_i)\cap S^k_{\epsilon,r}(A)}\theta_A(y,r_i)\le E-\eta$;\\
(3) if $r>0$, then $U_r=\bigcup_{i=1}^NB_r(x_i^r)$ with $N\le C(n)r^{-k}$;\\
(4) if $r=0$, then $U_0$ is $k$-rectifiable and satisfies $\text{Vol}(B_s(U_0))\le C(n) s^{n-k}$ for each $s>0$; in particular $H^k(U_0)\le C(n).$
\end{lemma}
The above lemma is proved by modifying the proof of Lemma 8.1 of \cite{NV17}, presented in Section 8 of \cite{NV17}. Thus we only give details for the modifications we made to \cite{NV17}, and for the rest steps that are identical (up to changing notations and terminologies) to those in \cite{NV17}, we refer the readers to the precise lines containing them, instead of letting the same details reappear in this paper. 
\begin{proof}
The proof starts to follow \cite{NV17} from the paragraph after remark 8.1 on p. 205, until the second paragraph of p. 211. Note in the second line of the proof of Lemma 8.5 of \cite{NV17}, the choice of $\tau(n)$ (see Definition \ref{effspa}) is determined. The first major difference comes in the third paragraph preceding Subsection 8.2, p. 211. Instead of applying Theorem 2.4 of \cite{NV17}, we apply the following Theorem herein:
\begin{theorem}[Quantitative dimension reduction]
\label{quandimred}
For each $\epsilon>0$, there exists $\delta(n,\Lambda,\epsilon)>0$, $r(n,\Lambda,\epsilon)>0$, such that if $A$ is $(k,\delta)$-symmetric in $B_2(p)$ with respect to some $k$-plane $V_k$, then for each $x\in B_1(p)\backslash B_{\epsilon}(V_k)$, then there exists some $r^{*}\ge r(n,\Lambda,\epsilon)$ such that $A$ is $(k+1,\epsilon)$-symmetric in $B_{r^*}(x)$.
\end{theorem} 
Note in Theorem \ref{quandimred} if we replace $A$ by any Radon measure $\mu$ obtained as in the paragraph preceding Definition \ref{conetipmu}, the conclusion still holds. This could be seen easily by using the weak-$*$ convergence $|F_{A_i}|^2d\text{Vol}_g\to \mu$. In spite of the almost identical statements of Theorem \ref{quandimred} and Theorem 2.4 of \cite{NV17}, their proofs (which are both by contradiction) slightly differ from each other due to the difference between the two notions of ``quantitative symmetry''. The proof of Theorem \ref{quandimred} will be given in Appendix B.

Upon applying Theorem \ref{quandimred}, the proof continues to follow \cite{NV17} until line 3, p.214. There we apply Theorem \ref{l2ap} instead of Theorem 7.1 of \cite{NV17}. Correspondingly, replace (8.38) of \cite{NV17} by
\begin{eqnarray}
\begin{split}
\label{model1}
D^k_{\mu^{\prime}}(y_j,s)\le C(n,\Lambda,\epsilon)s^{-k}\int_{B_s(y_j)}\bigg{(}W_s(z)+s^{2-n}\int_{B_{\epsilon_1 s}(z)}|\iota_{z-x}F_A|^2dV_g(x)\bigg{)}d\mu^{\prime}(z).\\
\end{split}
\end{eqnarray}
By applying this to all $r\le t\le s$, we have the following estimate in place of (8.39) of \cite{NV17}
\begin{eqnarray}
\begin{split}
\label{long1}
&s^{-k}\int_{B_s(x)}D^k_{\mu}(y,t)d\mu^{\prime}(y)\\
\le& Cs^{-k}\int_{B_s(x)}t^{-k}\int_{B_t(y)}\bigg{(}W_t(z)+t^{2-n}\int_{B_{\epsilon_1 t}(z)}|\iota_{z-x}F_A|^2dV_g(x)\bigg{)}d\mu^{\prime}(z)d\mu^{\prime}(y)\\
\le& Cs^{-k}t^{-k}\int_{B_s(x)}\mu^{\prime}(B_t(y))\bigg{(}W_t(z)+t^{2-n}\int_{B_{\epsilon_1 t}(z)}|\iota_{z-x}F_A|^2dV_g(x)\bigg{)}d\mu^{\prime}(z)\\
\le& Cs^{-k}\int_{B_s(x)}\bigg{(}W_t(z)+t^{2-n}\int_{B_{\epsilon_1 t}(z)}|\iota_{z-x}F_A|^2dV_g(x)\bigg{)}d\mu^{\prime}(z).
\end{split}
\end{eqnarray}
In above inequality, let $t=2^{-\beta}\le s\le 2^{-\alpha+1}$ and sum \eqref{long1} over all such $\beta$s. In place of (8.40) of \cite{NV17}, we obtain
\begin{eqnarray}
\begin{split}
\label{lll}
&\sum_{2^{-\beta}\le s}s^{-k}\int_{B_s(x)}D^k_{\mu^{\prime}}(y,2^{-\beta})d\mu^{\prime}(y)\\
\le& Cs^{-k}\int_{B_s(x)}\sum_{\overline{r}_y\le 2^{-\beta}\le s}\bigg{(}W_{2^{-\beta}}(z)+{(2^{-\beta})}^{2-n}\int_{B_{\epsilon_1 2^{-\beta}}(z)}|\iota_{z-x}F_A|^2dV_g(x)\bigg{)}d\mu^{\prime}(z)\\
\le& Cs^{-k}\bigg{(}\epsilon_1^{-1}\int_{B_s(x)}|\theta_A(z,4s)-\theta_A(z,\epsilon_1r)|d\mu^{\prime}(z)\\
&+\int_{B_s(x)}\bigg{(}\sum_{\overline{r}_y\le 2^{-\beta}\le s}{(2^{-\beta})}^{2-n}\int_{B_{\epsilon_1 2^{-\beta}}(z)}|\iota_{z-x}F_A|^2dV_g(x)\bigg{)}d\mu^{\prime}(z)\bigg{)}\\
\le& C(n,\Lambda,\epsilon)\epsilon_1^{-1}\eta^{\prime}+ Cs^{-k}\int_{B_s(x)}\bigg{(}\sum_{\overline{r}_y\le 2^{-\beta}\le s}{(2^{-\beta})}^{2-n}\int_{B_{\epsilon_1 2^{-\beta}}(z)}|\iota_{z-x}F_A|^2dV_g(x)\bigg{)}d\mu^{\prime}(z)\\
=&C(n,\Lambda,\epsilon)\epsilon_1^{-1}\eta^{\prime}+ II,
\end{split}
\end{eqnarray}
Now let us point out two distinctions between above estimate and (8.40) of \cite{NV17}. Firstly, the second term of the right hand side in \eqref{lll} does not exist in (8.40) of \cite{NV17}. In addition, compared to (8.40) of \cite{NV17}, an extra $\epsilon_1^{-1}$-factor appears in the first term on the right hand side of \eqref{lll} due to \eqref{w2drop}. To estimate $II$, let us use a dyadic decomposition to estimate following first:
\begin{eqnarray}
\begin{split}
\label{dyaest}
&\sum_{\overline{r}_y\le 2^{-\beta}\le s}{(2^{-\beta})}^{2-n}\int_{B_{\epsilon_1 2^{-\beta}}(z)}|\iota_{z-x}F_A|^2dV_g(x)\\
\le& \sum_{\overline{r}_y\le 2^{-\beta}\le s}{(2^{-\beta})}^{2-n}\sum_{k=0}^{\infty}\int_{A_{\epsilon_1 2^{-\beta-k-1},\epsilon_1 2^{-\beta-k}}(z)}|\iota_{z-x}F_A|^2dV_g(x)\\
\le& C(n)\sum_{\overline{r}_y\le 2^{-\beta}\le s}\sum_{k=0}^{\infty}{(2^{-k}\epsilon_1)}^{n-2}\int_{A_{\epsilon_1 2^{-\beta-k-1},\epsilon_1 2^{-\beta-k}}(z)}|x-z|^{4-n}|\iota_{\partial_{r_z}}F_A|^2dV_g(x)\\
=& C(n)\sum_{k=0}^{\infty}{(2^{-k}\epsilon_1)}^{n-2}\bigg{(}\sum_{\overline{r}_y\le 2^{-\beta}\le s}\int_{A_{\epsilon_1 2^{-\beta-k-1},\epsilon_1 2^{-\beta-k}}(z)}|x-z|^{4-n}|\iota_{\partial_{r_z}}F_A|^2dV_g(x)\bigg{)}\\
=&C(n)\sum_{k=0}^{\infty}{(2^{-k}\epsilon_1)}^{n-2}\bigg{(}\sum_{\overline{r}_y\le 2^{-\beta}\le s}\big{(}\theta_A(z,\epsilon_1 2^{-\beta-k-1})-\theta_A(z,\epsilon_1 2^{-\beta-k})\big{)}\bigg{)}\\
\le& 2C(n)\Lambda\sum_{k=0}^{\infty}{(2^{-k}\epsilon_1)}^{n-2}\\
\le& C(n,\Lambda)\epsilon_1^{n-2}.
\end{split}
\end{eqnarray}
In both the second equality and the penultimate inequality we used the monotonicity formula \eqref{mono}. Let us now insert \eqref{dyaest} into term $II$, \eqref{lll}, and use (8.37) of \cite{NV17} (note all inequalities preceding (8.38) of Section 8 of \cite{NV17} has been obtained in our context) to obtain
$$\sum_{\overline{r}_y\le 2^{-\beta}\le s}2^{\alpha k}\int_{B_s(x)}D^k_{\mu^{\prime}}(y,2^{-\beta})d\mu(y)\le C(n,\Lambda,\epsilon)\epsilon_1^{-1}\eta^{\prime}+C(n,\Lambda)\epsilon_1^{n-2}.$$
by firstly choosing $\epsilon_1\equiv\epsilon_1(n,\Lambda,\epsilon)$ and then choosing $\eta^{\prime}\le \eta^{\prime}(n,\Lambda,\epsilon)$ we proved (8.41) of \cite{NV17}; namely
\begin{eqnarray}
\begin{split}
\label{model2}
\sum_{2^{-\beta}\le s}s^{-k}\int_{B_s(x)}D^k_{\mu^{\prime}}(y,2^{-\beta})d\mu^{\prime}(y)\le \delta(n)^2,
\end{split}
\end{eqnarray}
where $\delta(n)$ is chosen from Theorem \ref{DRR}. By applying Theorem \ref{DRR} we obtain (8.42) of \cite{NV17} in our context, and hence the conclusion of Lemma \ref{keylemma} for the set $U_r$ ($r>0$) and $U_+$. To sum up, we have proved (1),
(2), (3) of Lemma \ref{keylemma}.

For (4) of Lemma \ref{keylemma}, namely the conclusion for $U_0$, we make modifications to (8.46), (8.47), and (8.48) of \cite{NV17} similar to the modifications we made to (8.38), (8.39), and (8.40) of \cite{NV17} respectively. Upon making these modifications and by following the same lines given in Subsection 8.3, pp. 214-216, \cite{NV17}, we also conclude (4) of Lemma \ref{keylemma}.
\end{proof}

\section{Proofs of the main theorems}\label{pomt}
In this section we complete the proofs of Theorems \ref{maintheorem}, \ref{maintheorem1} and \ref{maintheorem2}. Since the proofs are identical to that written in Section 9 of \cite{NV17} (up to switching a few notations from that context to this one), we omit the details. Nevertheless, for the readers' convenience we give an outline of how they are obtained from Lemma \ref{keylemma}. Here we follow Subsection 9.1 of [NV17].
\subsection{An outline}\label{outlineind}
The proof goes by induction on energy and iterative applications of Lemma \ref{keylemma}. To be precise, the aim is to show:
\begin{claim}
\label{indomit}
For each $l$, there exists a constant $C(l,n,\Lambda,\epsilon)$ and a covering 
\begin{eqnarray}
\begin{split}
S^k_{\epsilon,r}(A)\subseteq U_r^l\cup U^l_+=\bigcup_i B_{r}(x^{r,l}_i)\cup\bigcup_i B_{r_i^l}(x_i^l),
\end{split}
\end{eqnarray}
with $r_i^l>r$, such that the following two properties hold:
\begin{eqnarray}
\begin{split}
\label{indhold}
&r^{k-n}\text{Vol}(B_r(U_r^l))+\omega_k\sum_i (r_i^l)^k\le C(l,n,\Lambda,\epsilon),\\
&\sup_{y\in B_{r_i^l}(x_i^l)\cap S^k_{\epsilon,r}(A)}\theta_A(y,r_i^l)\le \Lambda-l\cdot\eta.
\end{split}
\end{eqnarray}
\end{claim}
The proof of Theorem \ref{maintheorem} will then follow from the claim. Indeed, let us take $l$ to be the least integer that $l\ge \frac{\Lambda}{\eta}+1$, and we apply above claim to such $l$. Then by the second property of \eqref{indhold} we see that $U_+^l=\emptyset$. Therefore $S^k_{\epsilon,r}(A)\subseteq U_r^l= \bigcup B_r(x_i^{l,r})$. Now by the first property of \eqref{indhold} we have $\text{Vol}(B_r(S^k_{\epsilon,r}(A)))\le C(l,n,\Lambda,\epsilon)\text{Vol}(B_r(U_r^l))\le C(n,\Lambda,\epsilon)r^{n-k}$. The proof of Claim \ref{indomit} goes by induction. Notice that the beginning stage $l=1$ follows from Lemma \ref{keylemma}. Thus let us assume we have proved Claim \ref{indomit} for some $l\ge 1$. To build the desired covering for stage $l+1$ we simply apply Lemma \ref{keylemma} by replacing $B_1(p)$ therein by $B_{r_i^l}(x_i^l)$ for each $i$. Then we obtain:
\begin{eqnarray}
\begin{split}
S^k_{\epsilon,r}\cap B_{r^l_i}(x^l_i)\subseteq U_{i,r}= \bigcup_l B_{r}(x_{i,l}^{r})\cup \bigcup_l B_{r_{i,l}}(x_{i,l}),\ \text{for all }i.
\end{split}
\end{eqnarray}
Now let us set $U_r^{l+1}=U_{r}^l\cup \bigcup_i U_{i,r}$ and $U_+^{l+1}=\bigcup_{i,l} B_{r_{i,l}}(x_{i,l})$. It is not hard to check that this is a desired cover for stage $l+1$. Thus finishes the proof of Claim \ref{indomit}. The Minkowski volume estimate \eqref{maininequality1} follows similarly. For the rest conclusions in the main theorems on the structure of the sets, they follow from Property (4) of Lemma \ref{keylemma} as well as standard geometric measure theory arguments; see the next subsection for their references.

\subsection{Proofs of Theorems \ref{maintheorem}, \ref{maintheorem1}, and \ref{maintheorem2}}
\begin{proof}
Up to necessary changes of notations, the proof of Theorem \ref{maintheorem} follows the lines presented in Subsection 9.1, pp. 216-218, \cite{NV17}, the proof of Theorem \ref{maintheorem1} follows the lines presented in Subsection 9.2, pp. 218-221, \cite{NV17}, and Theorem \ref{maintheorem2} follow the lines presented in Subsection 9.3, pp. 221-222, \cite{NV17}. Since all the proofs follow the corresponding arguments almost verbatim, we refer the readers to those lines in \cite{NV17} for details.
\end{proof}

\section{Application to the stationary Yang-Mills connections}\label{atsymc}
In this section, we will present an application of Theorems \ref{maintheorem}, \ref{maintheorem1}, and \ref{maintheorem2} to the admissible stationary Yang-Mills connections defined in Subsection 2.3 of \cite{T00}. For the readers' convenience, let us recall its definition. Let $A$ be a stationary connection on $M$. We say that $A$ is an admissible stationary Yang-Mills connection, if further there exists a closed subset $S(A)\subseteq M$ such that (1) $H^{n-4}(S(A)\cap K)<\infty$ for every compact set $K\subseteq M$; (2) $A$ is smooth and satisfies Yang-Mills equation outside $S(A)$.

For convenience let us assume the base manifold $M$ to be a topologically trivial geodesic ball $B_{16}(p)$ satisfying the condition at the beginning of Subsection \ref{mainresults}. Let $\{A_i\}_i$ be a sequence of admissible stationary Yang-Mills connections with $\|F_{A_i}\|_{L^2(B_{16}(p))}\le \Lambda$. Then up to a subsequence we could assume that $\{|F_{A_i}|^2d\text{Vol}_g\}$ converges as measures to some Radon measure $\mu$ in the weak-$*$ sense. By \cite{T00}, $\mu=|F_{A_{\infty}}|^2dV_g+\nu$, where $\nu$ (called the defect measure) is a positive Radon measure whose support $S$ has finite $H^{n-4}$-measure and is $n-4$-rectifiable; moreover, there exists a subsequence $\{i_l\}_l$ and a sequence of gauge transforms $\{\sigma_l\}_l$ that are smooth outside $S$, such that $\sigma_l^*A_{i_l}$ converges to $A_{\infty}$ smoothly outside $S$. From now on, we shall refer to this convergence as ``$A_i$ weakly converges to $(A_{\infty},\mu)$ with blow-up locus $S$''. We are interested in studying the stratification of the singular set of $A_{\infty}$. However, it might be the case that $A_{\infty}$ is no longer stationary. As a matter of fact, we might not be able to apply the theorems in the Subsection \ref{mainresults} directly to $A_{\infty}$. Nevertheless, we can prove results about a weaker stratification of $A_{\infty}$. Firstly, let us state the definition of tangent connection:
\begin{definition}[Tangent connection]
Let $A$ be an admissible stationary Yang-Mills connection. $A_{x_0}$ is called a tangent connection of $A$ at $x_0$, if there is a positive real number sequence $\lambda_i\to 0$ such that $A_{\lambda_i}(x)$ weakly converges to $(A_y,\mu)$; here $A_{\lambda_i}(x)\eqqcolon \lambda_i^{-1}A(\lambda_i(x-x_0)+x_0)$.
\end{definition}
A variant of Lemma 5.3.1 [G. Tian] shows that every tangent connection of $A_{\infty}$ is $0$-symmetric; in other words, it is  gauge equivalent to the pull back of a connection on the unit sphere $S^{n-1}$. For $A_{\infty}$, we have
\begin{definition}
For each $k=0,\cdots n-1$, the $k^{th}$ weak-stratum $\mathcal{W}\mathcal{S}^k(A_{\infty})$ is defined to be:
\begin{eqnarray}
\begin{split}
\mathcal{W}\mathcal{S}^k(A_{\infty})&\eqqcolon \bigg{\{}y\in B_1(p): \mbox{no tangent connection of $A_{\infty}$ at $y$ is $k+1$-symmetric}\bigg{\}}.
\end{split}
\end{eqnarray}
\end{definition}

\begin{theorem}
\label{coro1}
Let $\{A_i\}_i$ be a sequence of admissible stationary Yang-Mills connections with 
$$\int_{B_{16}(p)}|F_{A_i}|^2d\text{Vol}_g\le \Lambda,$$ 
such that $A_i$ weakly converges to $(A_{\infty},\mu_0)$. Then for each $k=0,\cdots,n-4$ and $\epsilon>0$, there exists a constant $C(n,\Lambda,\epsilon)$, such that for any $r>0$ we have
\begin{eqnarray}
\begin{split}
\label{mumin}
&\mbox{Vol}(B_r(S_{\epsilon}^{k}(\mu))\cap B_1(p))\le C(n,\Lambda,\epsilon)r^{n-k}.
\end{split}
\end{eqnarray}
Moreover, $\mathcal{W}\mathcal{S}^k(A_{\infty})\subseteq S^k(\mu)$, and $\mathcal{W}\mathcal{S}^k(A_{\infty}), S^k(\mu)$ are both $k$-rectifiable.
\end{theorem}
\begin{remark}
This theorem extends Proposition 3.3.3 of \cite{T00}. 
\end{remark}
\begin{proof}
Firstly, we need a covering lemma similar to Lemma 8.1 of \cite{NV17}, with exactly the same statement upon replacing $A$ by $\mu_0$:
\begin{lemma}
\label{keylemmamu}
Let $A$ be stationary connection satisfying the same conditions as Theorem \ref{maintheorem}. Let $E=\sup_{x\in B_1(p)\cap S^k_{\epsilon}(\mu_0)}\theta_{\mu_0}(x,1)$. Then for all $\eta \le \eta(n,\Lambda,\epsilon)$, there exists a covering $S^k_{\epsilon}(\mu)\cap B_1(p)\subseteq U= U_0\cup U_+$ such that:\\
(1) $U_+=\bigcup_i B_{r_i}(x_i)$ with $r_i> 0$ and $\sum_i r_i^k\le C(n,\Lambda, \epsilon)$;\\
(2) $\sup_{y\in B_{r_i}(x_i)\cap S^k_{\epsilon}(\mu_0)}\theta_{\mu_0}(y,r_i)\le E-\eta$;\\
(3) $U_0$ is $k$-rectifiable and satisfies $\text{Vol}(B_s(U_0))\le C(n) s^{n-k}$ for each $s>0$; in particular $H^k(U_0)\le C(n).$
\end{lemma}
The proof of this lemma follows almost the same arguments as that of Lemma 8.1 of \cite{NV17}, and property (2) follows trivially from the construction of the cover given after Remark 8.1 of \cite{NV17}. However there are two major modifications we need to make to the proof of Lemma 8.1 of \cite{NV17}. The first one is in the third paragraph preceding Subsection 8.2, p. 211. Instead of Theorem 2.4 of \cite{NV17}, we apply Theorem \ref{quandimred} (stated for $\mu_0$ instead of $A$) herein (see the paragraph after Theorem \ref{quandimred}). The second difference occurs when we apply the $L^2$-best approximation theorem in obtaining (1) and (3) of Lemma \ref{keylemmamu}. Let us take (3) for example, and note the same modifications apply to (1). If $\mu$ is $(0,\delta)$-symmetric in $B_{16}(p)$ but not $(k+1,\epsilon)$-symmetric in $B_1(p)$, then by the weak-$*$ convergence $|F_{A_i}|^2dV_g\to\mu_0$ we have for all $i$ sufficiently large that $A_i$ is $(0,2\delta)$-symmetric in $B_{16}(p)$ but not $(k+1,\epsilon/2)$-symmetric in $B_1(p)$. Now let us fix arbitrary sufficiently large integer $K_*\gg 1$. Then, in the line before (8.46) of \cite{NV17}, we apply Theorem \ref{l2ap} with $A$, $B_1(p)$, $\epsilon$, and $\delta$ replaced by $A_i$, $B_{2^{-\beta}}(y)$, $\epsilon/2$, and $2\delta$ respectively, where $i\ge i(K_*)$ is sufficiently large, $2^{-K_*}\le 2^{-\beta}\le s$, with $y$, $s$ and $\mu$ are the same as in (8.48) of \cite{NV17}. We then obtain:
\begin{eqnarray}
\begin{split}
\label{ppp}
D^k_{\mu}(y,2^{-\beta})\le C(n,\Lambda,\epsilon)(2^{\beta})^{-k}\int_{B_{2^{-\beta}}(y_j)}\bigg{(}W_s(z)+(2^{-\beta})^{2-n}\int_{B_{\epsilon_1 2^{-\beta}}(z)}|\iota_{z-x}F_{A_i}|^2dV_g(x)\bigg{)}d\mu(z).
\end{split}
\end{eqnarray}
Then we take the sum of \eqref{ppp} over all such $\beta$ and follow the same estimates \eqref{long1}-\eqref{model2} to obtain
\begin{eqnarray}
\begin{split}
\sum_{2^{-K_*}\le 2^{-\beta}\le s}s^{-k}\int_{B_s(x)}D^k_{\mu}(y,2^{-\beta})d\mu(y)\le \delta(n)^2.
\end{split}
\end{eqnarray}
By the arbitrariness of $K_*$, we therefore have
\begin{eqnarray}
\begin{split}
\sum_{2^{-\beta}\le s}s^{-k}\int_{B_s(x)}D^k_{\mu}(y,2^{-\beta})d\mu(y)\le \delta(n)^2.
\end{split}
\end{eqnarray}
Now we can apply Theorem \ref{RR} to see that $U_0$ is rectifiable. Thus finishes the proof of (3) of Lemma \ref{keylemmamu}. The same modifications given above apply to conclude (1) as well. After proving Lemma \ref{keydefmu}, we could then follow exactly the same arguments in Subsections 9.1 and 9.2 of \cite{NV17} to obtain the rectifiability of $S^k_{\epsilon}(\mu_0)$ as well as \eqref{mumin}.

Now it remains to show $\mathcal{W}\mathcal{S}^k(A_{\infty})\subseteq S^k(\mu)$. It suffices to prove $\mathcal{W}\mathcal{S}^k(A_{\infty})^c\supseteq S^k(\mu)^c$. Let $x_0$ be a point at which $\mu$ has a $(k+1)$-symmetric tangent measure $\eta$. By definition, there exist positive real number sequences $\lambda_l\to0$ and $\epsilon_l\to0$ such that for each $i$, there exist $(\epsilon_l,1)$-cone tips of $\mu_{\lambda_l}$ in $B_{1/2}(x_0,\lambda_l^{-2}g)$, denoted by $x_0,p_1^{(i)},\cdots,p_{k+1}^{(i)}$, that $\tau(n)$-effectively span a $k+1$-plane; up to passing to subsequence we may assume that $p_\alpha^{(i)}\to p_\alpha$ as $i\to\infty$ for each $\alpha=1,\cdots,k+1$, and that they span some $k+1$-plane $V_{k+1}$. Further by the weak-$*$ convergence $|F_{A_i}|^2dV_g\to\mu$, there exists a subsequence $\{A_{i_l}\}_{l}$ such that $|F_{A_{i_l,\lambda_{l},x_0}}|^2d\text{Vol}_{\lambda_l^{-1}g}\to \eta$ in the weak-$*$ sense. By Theorem 3.1.2 and Theorem 3.3.3 of \cite{T00}, $A_{i_l,\lambda_l}$ converge weakly to some pair $(A_{x_0},\tilde{\nu})$ with blow-up locus $\tilde{S}$, such that $\tilde{S}$ is $(n-4)$-rectifiable, $\tilde{\nu}=\tilde{\Theta} dH^{n-4}\big{\rvert}_{\tilde{S}}$ for some positive function $\tilde{\Theta}$, and $\eta=|F_{A_{x_0}}|^2d\text{Vol}+\tilde{\nu}$. On the one hand, by (5.3.4) of \cite{T00} we achieve the following identity for all $s>0$ and each $\alpha=0,\cdots,k+1$:
\begin{eqnarray}
\begin{split}
\label{fatou}
&\int_{\tilde{S}\cap A_{2\epsilon s,2s}(p_{\alpha})}|x-p_\alpha|^{4-n}|\nabla^{\bot}r_\alpha|^2\tilde{\Theta} dH^{n-4}(x)+4\int_{A_{2\epsilon s,2s}(p_\alpha)}|x-p_\alpha|^{4-n}|\iota_{\partial r_{\alpha}}F_{A_{x_0}}|^2dV_g\\
=&\theta_{\eta}(p_\alpha,2s)-\theta_{\eta}(p_\alpha,2\epsilon s)=0,
\end{split}
\end{eqnarray}
where $\nabla^{\bot}r_{p_\alpha}$ denotes the component of $\nabla r_{p_\alpha}$ that is perpendicular to the tangent space of $\tilde{S}$; by the $(n-4)$-rectifiability of $\tilde{S}$, $\nabla^{\bot}r$ is $H^{n-4}$-a.e. well defined on $S$; the last identity in \eqref{fatou} follows from the $(k+1)$-symmetry of $\eta$. Especially, this implies that:
\begin{equation}
\label{fatou0}
\iota_{\partial r_{p_\alpha}}F_{A_{x_0}}(x)=0,\ \text{for all } x\in \mathbb{R}^n, \alpha=0,\cdots,k+1.
\end{equation}
Therefore $A_{x_0}$ is a $(k+1)$-symmetric connection. On the other hand, due to the facts that $|F_{A_i}|^2dV_g\to\mu=|F_{A_{\infty}}|^2d\text{Vol}_g+\nu$, $\mu_{\lambda_l}\to \eta$ both in the weak-$*$ sense, $A_{x_0}$ is a tangent connection of $A_\infty$ at $x_0$; in other words, $x_0\in \mathcal{W}\mathcal{S}^k(A_{\infty})^c$. By the arbitrariness of $x_0$, we conclude that $\mathcal{W}\mathcal{S}^k(A_{\infty})\subseteq S^k(\mu)$. The rectifiability of $\mathcal{W}\mathcal{S}^k(A_{\infty})$ hence follows immediately from that of $S^k(\mu)$. Thus we complete the proof of Theorem \ref{coro1}.
\end{proof}

\section*{Appendix A: proof of Claim \ref{introproof}}\label{appa}
\begin{proof}
We may assume that $\mu$ is a tangent measure of a stationary connection $A$ at $p$. In other words, $|F_{A_l}|^2d\text{Vol}\to\eta$ in the weak-$\ast$ sense, where $A_l(x)\eqqcolon \lambda_l^{-1} A(\lambda_l(x-p))$. According to the discussions in the paragraph preceding Definition \ref{symdef1}, we may further assume that $\mu$ is a cone measure at the origin and set
$$\theta_{\mu}(y,r)= r^{4-n}\mu(B_r(y)).$$
Firstly, assume that $\mu$ is $k$-symmetric in the sense of Definition \ref{symdef1}. In other words, $\mu$ satisfies $T_{e_i}^*\mu=\mu$ for $i=1,\cdots,k$ where $\{e_i\}_{i=1}^n$ forms an orthonormal basis. As a matter of fact, $\theta_{\mu}(x,r)\eqqcolon r^{4-n}\mu(B_r(x))$ is constant in $r$ for all $x\in \text{span}\{e_i\}_{i=1}^k$. Apparently this implies that $\mu$ is $k$-symmetric in the sense of Definition \ref{keydefmu}. This completes the proof of one direction.

For the converse, let us assume that $\mu$ is $k$-symmetric in $B_1(0)$ with respect to some $k$-plane $V_k$ in the sense of Definition \ref{keydefmu}. Suppose that $V_k$ is spanned by an orthonormal basis $\{e_i\}_{i=1}^k$. Also denote by $z_1,\cdots,z_k$ the points such that $z_i=e_i/4$. From Remark \ref{keyremark} we have
$$\lim_{l\to \infty}\sum_{j=1}^k\int_{B_1(p)}|\iota_{\partial z_j}F_{A_{l}}|^2(x)dV_{\lambda_l^{-2}g}(x)=0,$$
$$\lim_{l\to \infty}\int_{A_{\sigma,\rho}(p)}|x-p|^{4-n}|\iota_{\partial r_p}F_{A_{l}}|^2(x)dV_{\lambda_l^{-2}g}(x)=0,\ \mbox{for all $\sigma,\ \rho\le 1/2$.}$$
Using this and \eqref{virtue}, we have the following for all $y\in B_{1/2}(0)\cap V_k$:
\begin{eqnarray}
\begin{split}
\lim_{l\to \infty}\int_{A_{\sigma,\rho}(y)}|x-y|^{4-n}|\iota_{\partial r_y}F_{A_{l}}|^2(x)dV_{\lambda_l^{-2}g}(x)=0,\ \mbox{for all $\sigma,\ \rho\le 1/2$.}
\end{split}
\end{eqnarray}
Hence for all $y\in B_{1/2}(0)\cap V_k$ and all $\sigma,\rho\le1/2$, we obtain:
\begin{eqnarray}
\begin{split}
\label{halfballrigid}
\theta_{\mu}(y,\rho)-\theta_{\mu}(y,\sigma)&=\rho^{4-n}\mu(B_{\rho}(y))-\sigma^{4-n}\mu(B_{\sigma}(y))\\
&=\lim_{l\to\infty}\int_{A_{\sigma,\rho}(y)}4|x-y|^{4-n}|\iota_{\partial r_y}F_{A_{l}}|^2dV_{\lambda_l^{-2}g}(x)\\
&=0.
\end{split}
\end{eqnarray}
Now we show that $\theta_{\mu}(q,r)$ is constant in $r$ for $q=0,z_1,\cdots,z_k$. Choose any $r>0$. By the fact that $\eta$ is a cone measure at $0$, we have the following for any $\epsilon>0$:
$$\theta_{\mu}(z_i,r)-\theta_{\mu}(z_i,1/2)=\theta_{\mu}(\epsilon z_i,\epsilon r)-\theta_{\mu}(\epsilon z_i,\epsilon/2).$$
Choose $\epsilon$ small enough to guarantee that $B_{\epsilon r}(\epsilon z_i)\subseteq B_{1/2}(p)$. Upon applying \eqref{halfballrigid}, we see that the right hand side of the above equation is $0$. Therefore, we have $\theta_{\mu}(z_i,r)-\theta_{\mu}(z_i,1/2)=0$. Hence $\theta_{\mu}(q,r)$ is constant in $r$ for $q=0,z_1\cdots,z_k$. By \eqref{mono}, we have:
$$\lim_{l\to\infty}\int_{A_{\sigma,\rho}(z_i)}|\iota_{\partial r_{z_i}}F_{A_{l}}|^2(x)dV_{\lambda_l^{-2}g}(x)=0,\ \mbox{for all $\sigma,\ \rho>0$}$$
To see that this implies $\mu$ is a cone measure at $0,z_1,\cdots,z_k$, we apply \eqref{psisph} for any $\sigma,\rho>0$ and any radial symmetric function $\psi$ at $q$ for $q=0,z_1,\cdots,z_k$ to obtain
\begin{eqnarray}
\begin{split}
\label{sph}
&\rho^{4-n}\int_{B_{\rho}(q)}\psi d\mu-\sigma^{4-n}\int_{B_{\sigma}(q)}\psi d\mu\\
=&\lim_{l\to\infty}\int_{A_{\sigma,\rho}(z_i)}4|x-z_i|^{4-n}\psi|\iota_{\partial r_{z_i}}F_{A_{l}}|^2dV_{\lambda_l^{-2}g}\\
&-\lim_{\lambda_l\to0}\int_{\sigma}^{\rho}4\tau^{3-n}\big{(}\int_{B_{\tau}(z_i)}|x-z_i|\langle \iota_{\partial r_{z_i}}F_{A_{l}},\iota_{\nabla\psi}F_{A_{l}} \rangle  dV_{\lambda_l^{-2}g}\big{)} d\tau\\
=&0.
\end{split}
\end{eqnarray}
Since $\psi$ is arbitrary, \eqref{sph} implies that $\mu$ is a cone measure at $0,z_1,\cdots,z_k$, and (if we denote $0$ by $z_0$) for $i=0,\cdots,k$ we have $d\mu=r_i^{n-5}dr_{z_i}\wedge d\xi_i(\theta)$, where $r_{z_i}(\cdot)=\text{dist}(z_i,\cdot)$ and $d\xi_i(\theta)$ is a Radon measure on the unit sphere $\{z\in T_pM: r_{z_i}(z)=1\}$. Clearly it follows that $\mu(z_1,\cdots, z_n)=\mu(z_{k+1},\cdots,z_n)$. In other words, $\mu$ is $k$-symmetric in the sense of Definition \ref{symdef1}. This completes the proof of the other direction, and hence concludes Claim \ref{introproof}. 
\end{proof}

\section*{Appendix B: proof of Theorem \ref{quandimred}}\label{appb}
\begin{proof}
The proof is by contradiction. Suppose the contrary holds. Then there exists an $\epsilon_0$, a sequence of stationary connections $A_i$, positive real number sequences $\delta_i\rightarrow 0$ and $r_i\rightarrow 0$ such that $A_i$ is $(k,\delta_i)$-symmetric in $B_1(x_i)$ with respect to $V_k^{(i)}$, and $y_i\in B_1(x_i)\backslash B_{\epsilon_0 }(V_k^{(i)})$ where no such $r^{*}>r_i$ exists that $A_i$ is $(k+1,\epsilon_0)$-symmetric in $B_{r^{*}}(y_i)$. Assume $V_k^{(i)}$ is $1/2$-effectively spanned by $0,\xi_{1}^{i},\cdots,\xi_{k}^{i}$.

Up to passing to a subsequence we may assume $x_i=0\in \mathbb{R}^n$, $y_i\rightarrow y_{\infty}$、 and $\xi_{\alpha}^{(i)}\to \xi_{\alpha}^{(\infty)}$ for each $\alpha=1,\cdots,k$; denote by $V_{k}^{\infty}$ the $k$-plane spanned by $0,\xi^{\infty}_1,\cdots,\xi^{\infty}_k$. For convenience, let us denote $0$ by $\xi^{i}_{0}$ for all $i$. Moreover, assume that $|F_{A_{i}}|^2dV_g$ converges to a measure $\mu$ in weak-$\ast$ sense. Thus, the fact that $A_i$ is $(k,\delta_i)$-symmetric in $B_1(0)$ with respect to $V_k^{(i)}$ implies that $\mu$ is $k$-symmetric in $B_1(0)$ with respect to $V_k^{\infty}$.

Let us choose $r_0$ sufficiently small to be determined later. On the one hand, by the monotonicity of $r^{4-n}\mu(B_r(y_{\infty}))$ and the pigeonhole principle, for any $r_0$ there exists $r^*$ with $r_0\ge r^*\ge r_0\cdot(\epsilon_0/10)^{10\Lambda/\epsilon_0}\equiv r_0\cdot r(n,\Lambda,\epsilon_0)$ such that
\begin{eqnarray}
\begin{split}
\label{model31}
|\theta_A(y_{\infty},10r^*)-\theta_A(y_{\infty},\epsilon_0 r^*/10)|<\epsilon_0/10.
\end{split}
\end{eqnarray}
On the other hand, let $V_k^{(y_{\infty})}$ be the $k$-plane that is parallel to $V_k^{(\infty)}$ and passing through $y_{\infty}$. We may find $k$ points $\{y_{\infty,l}\}_{l=1}^k\subseteq B_{r^*}(y_{\infty})\cap V_k^{(\infty)}$ such that together with $y_{\infty}$ they $\tau(n)r^*$-effectively span $V_k^{(\infty)}$ in $B_{r^*}(y_{\infty})$. By the $k$-symmetry of $\mu$ in $B_1(0)$, we have
\begin{eqnarray}
\begin{split}
\label{model32}
|\theta(y_{\infty,l},10r^*)-\theta(y_{\infty,l},\epsilon_0 r^*/10)|<\epsilon_0/10,\ l=0,\cdots,k,
\end{split}
\end{eqnarray}
where we denoted $y_{\infty}$ by $y_{\infty,0}$. Let $y_{\infty}^{\text{proj}}$ be the projection image of $y_{\infty}$ onto $V_k^{\infty}$. Denote by $y_{r^*}$ the intersecting point of $\partial B_{r^*}(y_{\infty})$ with the line passing through $y_{\infty}^{\text{proj}}$ and $y_{\infty}$, and set $d^*=\text{dist}(y_{\infty}^{\text{proj}},y_{\infty})$. Using the fact that $r^{4-n}\mu(B_r(y_{\infty}^{\text{proj}}))$ is constant in $r\in [0,1/2]$, we have
\begin{eqnarray}
\begin{split}
\label{model34}
t^{4-n}\mu(B_t(y_{r_*}))=s^{4-n}\mu(B_s(y_{\infty})), \text{where }\frac{s}{t}=\frac{d^*}{d^*+r^*}.
\end{split}
\end{eqnarray}
By \eqref{model31}, \eqref{model34}, the fact that $d^*\ge\epsilon_0$, and choosing a sufficiently small $r_0\le r_0(\epsilon_0)$, we obtain
\begin{eqnarray}
\begin{split}
\label{model33}
|\theta(y_{r^*},s)-\theta(y_{r^*},\epsilon_0s)|<\epsilon_0/10,\ \epsilon_0r^*/5\le s\le5r^*.
\end{split}
\end{eqnarray}
From \eqref{model32} and \eqref{model33}, we see that $y_{r^*},y_{\infty},y_{\infty,1},\cdots,y_{\infty,k}$ are $(\epsilon_0/2,r_*)$-cone points at $y_{\infty}$; moreover, they $2\tau(n)r_*$-effectively span a $(k+1)$-plane. By Definition \ref{keydefmu}, $\mu$ is $(k+1,\epsilon_0/2)$-symmetric on $B_{r^*}(y_{\infty})$. By the weak-$*$ convergence, $A_i$ is $(k+1,\epsilon_0)$-symmetric on $B_{r^*}(y_{i})$ for all $i$ sufficiently large. This gives a contradiction since $r^*>r_i$ for all $i$ sufficiently large. Thus we complete the proof of Theorem \ref{quandimred}. 
\end{proof}

\section*{Acknowledgements}
The author gratefully thanks his advisor A. Naber for interesting him in the problem and giving him constant support and tremendous encouragements.



\bibliographystyle{plain}
\bibliography{quanref}


\end{document}